\documentclass[12pt]{amsart}
\usepackage{pifont}
\usepackage{txfonts}
\usepackage{}
\usepackage{amsfonts}
\usepackage{graphicx}
\usepackage{epstopdf}
\usepackage{amsmath}
\usepackage{amsthm}
\usepackage{latexsym,bm}
\usepackage{amssymb}
\usepackage{esint}
\usepackage{enumitem}
\usepackage{indentfirst}
\usepackage{amscd}
\newtheorem{thm}{Theorem}[section]

\newtheorem{lem}[thm]{Lemma}

\newtheorem{defn}[thm]{Definition}

\newtheorem{Remark}{Remark}
\numberwithin{equation}{section}
\numberwithin{Remark}{section}

\begin{document}

\title{A normalized Ricci flow on surfaces with boundary towards the complete hyperbolic metric}

\renewcommand{\subjclassname}{\textup{2000} Mathematics Subject Classification}

\begin{abstract}
Let $(\overline{M},g_0)$ be a $2$-D compact surface with boundary $\partial M$ and its interior $M$. We show that for a large class of initial and boundary data, the initial-boundary value problem of the normalized Ricci flow $(\ref{equn_RF3})-(\ref{equn_ibp3})$, with prescribed geodesic curvature $\psi$ on $\partial M$, has a unique solution for all $t>0$, and it converges to the complete hyperbolic metric locally uniformly in $M$. Here the natural condition that $\psi>0$ causes the main difficulty in the a priori estimates in the corresponding initial-boundary problem $(\ref{equn_RF3-1})-(\ref{equn_ibp3-1})$ of the parabolic equations, for which an auxiliary Cauchy-Dirichlet problem is introduced. We also provide examples of the boundary data $\psi$ which fits well with the natural asymptotic behavior of the geodesic curvature, but the solution to $(\ref{equn_RF3})-(\ref{equn_ibp3})$ fails to converge to the complete hyperbolic metric.
\end{abstract}

\renewcommand{\subjclassname}{\textup{2000} Mathematics Subject Classification}


\thanks{$^\dag$ Research supported by the National Natural Science Foundation of China No. 11701326.}

\address{Gang Li, Department of Mathematics, Shandong University, Jinan, Shandong Province, China}
\email{runxing3@gmail.com}

\maketitle


\section{Introduction}

 Let $M$ be the interior of a compact surface $\overline{M}$ with boundary $\partial M$, and $\bar{g}$ be a Riemannian metric on $\overline{M}$. We denote by $K$ the Gauss curvature of $\bar{g}$ and $k$ be the geodesic curvature of $\partial M$. For a closed surface $\Sigma^2$, R. Hamilton \cite{Hamilton1}\cite{Hamilton2} introduced the Ricci flow
 \begin{align}\label{equn_RF}
 \frac{\partial}{\partial t}\bar{g}=-2(K-\overline{K})\bar{g},
 \end{align}
with $\overline{K}$ the mean value of the Gauss curvature $K$ on $(\Sigma^2, \bar{g})$. It has been proved by Hamilton \cite{Hamilton2} and Chow \cite{Chow} that for the initial value problem of the Ricci flow on a closed surface starting from every initial metric, there exists a unique solution defined for all $t\geq 0$, and moreover, it converges exponentially to a metric of constant Gauss curvature as $t\to\infty$. Ricci flow on surfaces with boundary has been studied by many authors, see \cite{TLi}\cite{Brendle1}\cite{CM}, and for some generalized curvature flows, see \cite{Brendle2}\cite{Zhang}. In \cite{Brendle1}, based on the variational method, Brendle considered the initial-boundary value problem of $(\ref{equn_RF})$ on $\overline{M}$ with the boundary condition
\begin{align}\label{equn_bddzero}
k=0,
\end{align}
and showed that for every initial metric, there exists a unique solution to the boundary value problem $(\ref{equn_RF})-(\ref{equn_bddzero})$ defined for all $t\geq0$; moreover, the flow converges exponentially to a metric with constant Gauss curvature in $\overline{M}$ and vanishing geodesic curvature on $\partial M$. He also considered another boundary value problem
\begin{align*}
& \frac{\partial}{\partial t}\bar{g}=-2(k-\overline{k})\bar{g},\,\,\text{on}\,\partial M,\\
&K=0,\,\,\text{in}\,\,M,
\end{align*}
with $\overline{k}$ the mean value of the geodesic curvature $k$, and obtain the same convergence result of the flow to a metric with constant geodesic
curvature and vanishing Gauss curvature. In \cite{CM}, the authors considered the initial-boundary value problem on $(\overline{M}^2,g_0)$
\begin{align}\label{equn_RF2}
&\frac{\partial}{\partial t}\bar{g}=-2K\bar{g},\,\text{in}\,M\times [0,T)\,\\
&k=\psi(\cdot,t),\,\,\text{on}\,\partial M \times [0,T),\\
&\label{equn_ibp2}\bar{g}\big|_{t=0}=g_0,\,\,\text{in}\,M,
\end{align}
with the prescribed geodesic curvature function $\psi$ on $\partial M \times (0,T)$. They showed that when the initial data satisfies $K_{g_0}>0$ in $\overline{M}$ and with constant geodesic curvature $k_{g_0}\geq0$, and $\psi\geq 0$ is constant in space and non-increasing in $t$ with $\psi\big|_{t=0}=k_{g_0}$, then the normalized flow corresponding to $(\ref{equn_RF2})-(\ref{equn_ibp2})$ converges to a metric of constant curvature in $\overline{M}$ with totally geodesic boundary. 
 On the other hand, they showed that when $g_0$ is radially symmetric on the disk in $\mathbb{R}^2$, with $k_{g_0}\leq0$ on the boundary with $\psi\equiv k_{g_0}$, then the solution to the normalized flow corresponding to $(\ref{equn_RF2})-(\ref{equn_ibp2})$ exists for all time. For high dimensional compact manifolds with boundary, the research on the boundary value problem of the Ricci flow is referred to \cite{Shen,Cortissoz,Pulemotov,Gianniotis1,Gianniotis2,TChow,KS,Jouttijarvi}, and it is far from complete.

Let $\overline{M}$ be a bounded domain in $\mathbb{R}^n$ with interior $M$ and smooth boundary $\partial M$ with $n\geq3$. In \cite{LN}, Loewner and Nirenberg considered the following boundary value problem (now called the Loewner-Nirenberg problem)
\begin{align}\label{equn_LNF}
&\frac{4(n-1)}{(n-2)}\Delta u=n(n-1)u^{\frac{n+2}{n-2}},\,\,\text{in}\,\,M,\\
&\label{equn_LNBD}u(p)\to+\infty,\,\,\text{as}\,\,p\to\partial M.
\end{align}
Using the maximum principle, they showed the monotonicity of the sequence of solutions $\{u_k\}_{k=1}^{\infty}$ to the corresponding Dirichlet boundary value problems of $(\ref{equn_LNF})$ with boundary data $u=k$ for $k\in \mathbb{N}$, and that $u_k$ converges locally uniformly to the unique solution of the Loewner-Nirenberg problem $(\ref{equn_LNF})-(\ref{equn_LNBD})$ in $M$ as $k\to\infty$. The geometric meaning of the Loewner-Nirenberg problem is that, for any compact domain $\overline{M}$ in the Euclidean space, there exists a unique conformal metric $h=u^{\frac{4}{n-2}}g$ which is complete in $M$, with scalar curvature $-n(n-1)$.  Later Aviles and McOwen \cite{AM,AM1} generalized this result to compact Riemannian manifolds with boundary (for uniqueness of the solution, see \cite{ACF}\cite{LM2}). For the regularity and expansion of the solution near $\partial M$, one refers to \cite{Mazzeo,ACF,BE,Kichenassamy,HJ}. 
 For $n=2$, on a compact surface $(\overline{M},g)$ with its interior $M$ and boundary $\partial M$, the Loewner-Nirenberg problem reduces to
\begin{align}\label{equn_LNFn2}
&-\Delta_gu+K_g=-e^{2u},\\
&\label{equn_LNBDn2}u(p)\to+\infty,\,\,\text{as}\,\,p\to\partial M.
\end{align}
This problem refers to the uniformization theorem, and it was first settled by Bieberbach \cite{Bieberbach}, see \cite{LM} when $\partial M$ has less regularity. One can run the approach of Loewner and Nirenberg to solve the problem directly, and we call the solution the Loewner-Nirenberg solution, denoted as $u_{LN}$. Using the same argument on the regularity of $u_{LN}$ near $\partial M$ as the Loewner-Nirenberg problem in higher dimensions, we have $e^{u_{LN}}-r^{-1}=O(r)$ as $r\to0$, where $r$ is the distance function to $\partial M$ in $(\overline{M},g)$, see \cite{HS}.

For the Loewner-Nirenberg problem on a compact manifolds with boundary of dimension $n\geq3$, inspired by \cite{LN,AM,AM1}, we \cite{GLi} performed the Cauchy-Direchlet problems of two geometric flow approaches and obtain the convergence of the flow to the solution of the Loewner-Nirenberg problem when the boundary data $\phi\to +\infty$ in a certain speed as $t\to\infty$.

In this paper, on a compact surface $(\overline{M},g_0)$ with boundary $\partial M$, we study a normalized Ricci flow approach to $(\ref{equn_LNFn2})-(\ref{equn_LNBDn2})$ with prescribed geodesic curvature on $\partial M$
\begin{align}\label{equn_RF3}
&\frac{\partial}{\partial t}\bar{g}=-2(K_{\bar{g}}+1)\bar{g},\,\text{in}\,\overline{M}\times [0,\infty)\,\\
&k_{\bar{g}}=\psi(\cdot,t),\,\,\text{on}\,\partial M \times [0,\infty),\\
&\label{equn_ibp3}\bar{g}\big|_{t=0}=g_0,\,\,\text{in}\,\overline{M},
\end{align}
with $\bar{g}=e^{2v}g_0$ and $K_{\bar{g}}$ the Gauss curvature of the metric $\bar{g}$ and $k_{\bar{g}}$ the geodesic curvature of $\partial M$ in $(\overline{M},\bar{g})$. Recall that under the conformal change $\bar{g}=e^{2u}g$, we have
\begin{align}
&-\Delta_{g}u+K_{g}=K_{\bar{g}}e^{2u},\\
&\frac{\partial}{\partial n_{g}}u+k_{g}=k_{\bar{g}}e^u,
\end{align}
where $k_{g}$ and $k_{\bar{g}}$ are the geodesic curvatures of $\partial M$ with respect to $g$ and $\bar{g}$, and $n_{g}$ is the outer unit normal vector field of $\partial M$ in $(\overline{M},g)$. Therefore, with a background metric $g$ conformal to $g_0$ and denoting $\bar{g}=e^{2u}g$, the initial-boundary value problem $(\ref{equn_RF3})-(\ref{equn_ibp3})$ can be rewritten as
\begin{align}\label{equn_RF3-1}
&u_t=e^{-2u}(\Delta_{g}u-K_{g})-1,\,\text{in}\,\overline{M}\times [0,\infty)\,\\
&\label{equn_bp3-1}\frac{\partial}{\partial n_{g}}u+k_{g}=\psi e^u,\,\,\text{on}\,\partial M \times [0,\infty),\\
&\label{equn_ibp3-1}u\big|_{t=0}=u_0,\,\,\text{in}\,\overline{M}.
\end{align}
We wonder for what kind of initial data $u_0$ and boundary data $\psi$ the flow $(\ref{equn_RF3-1})-(\ref{equn_ibp3-1})$ converges to the Loewner-Nirenberg metric locally uniformly in $M$. Notice that we are able to use comparison theorem and solutions of Loewner-Nirenberg problem on geodesic balls to obtain a uniform upper bound of the solution to the flow $(\ref{equn_RF3-1})-(\ref{equn_ibp3-1})$ on each compact subset of $M$; while in general there is no uniform lower bound estimates of $u$.

 There are two main difficulties for this problem: one is to figure out the suitable region of the boundary data $\psi$, for which one will find that we need to take the values of $\psi$ in a certain positive interval to fit the behavior of the solution to the Loewner-Nirenberg problem $(\ref{equn_LNFn2})-(\ref{equn_LNBDn2})$ for $t$ large; but then the other big challenge is the a priori estimates near $\partial M$, provided that now the term $\psi$ in $(\ref{equn_bp3-1})$ is positive for $t$ large, and there is few result on a priori estimates near the boundary in this case for second order parabolic equations. Geometrically, that causes difficulties in the estimates on the derivatives of the curvatures when performing Hamilton's approach in \cite{CM}, which forced them to assume that $g_0$ is radial symmetric for a long time existence result. In fact, we provide a counter example, which says that even if $\psi$ converges asymptotically to the right data as $t\to+\infty$, the flow could diverge and go further and further away from the Loewner-Nirenberg metric.

 To handle the initial-boundary problem, we introduce an auxiliary problem: the Cauchy-Dirichlet problem of the normalized Ricci flow:
 \begin{align}\label{equn_RF3-2}
&u_t=e^{-2u}(\Delta_{g}u-K_{g})-1,\,\text{in}\,M\times [0,\infty)\,\\
&u=\phi(\cdot,t),\,\,\text{on}\,\partial M \times [0,\infty),\\
&\label{equn_ibp3-2}u\big|_{t=0}=u_0,\,\,\text{in}\,M,
 \end{align}
which is a generalization of the Cauchy-Dirichlet problem of the Yamabe flow in \cite{GLi}.  We establish a comparison theorem (see Theorem \ref{thm_comparison2}) for solutions to $(\ref{equn_RF3-1})-(\ref{equn_ibp3-1})$, and use solutions to the Cauchy-Dirichlet problem with different Dirichlet boundary data $\phi$ to control the upper and lower bounds of the solution $u$ to $(\ref{equn_RF3-1})-(\ref{equn_ibp3-1})$ to get long time existence and convergence of $u$.

To obtain a solution $u\in C^{2+\alpha,1+\frac{\alpha}{2}}(M\times[0,T])$ to $(\ref{equn_RF3-2})-(\ref{equn_ibp3-2})$, we need the following compatibility condition:
\begin{align}\label{equn_CDPcompatcd1}
&u_{0}(x)=\phi(x,0),\,\,\text{for}\,\,x\in \partial M,\\
&\label{equn_CDPcompatcd1-1}\frac{\partial \phi}{\partial t}(x,0)=e^{-2u_{0}(x)}(\Delta_gu_{0}(x)-K_{g})-1,\,\,\text{for}\,\,x\in\,\partial M.
\end{align}
We have the following existence result for the Cauchy-Dirichlet problem, see Section \ref{section_convergCDP}.
\begin{thm}\label{thm_CDPconvergnl}
Let $(\overline{M},g)$ be a compact surface with its interior $M$ and boundary $\partial M$. Assume that $u_0\in C^{2+\alpha}(\overline{M})$ and $\phi\in C^{2+\alpha,1+\frac{\alpha}{2}}(\partial M\times[0,T])$ for all $T>0$, and also the compatibility condition $(\ref{equn_CDPcompatcd1})-(\ref{equn_CDPcompatcd1-1})$ holds. Moreover, we assume that $\phi_t(x,t)\geq0$ on $\partial M\times[T,\infty)$ for some $T>0$, and $\phi(x,t)\geq \log(\xi(t))$ for $(x,t)\in \partial M\times[T_1,+\infty)$ with some constant $T_1>0$ where $\xi$ is a low-speed increasing function satisfying $(\ref{inequn_lowspeedinc})$. Then there exists a unique solution $u\in C^{2+\alpha,1+\frac{\alpha}{2}}(\overline{M}\times[0,T'])$ for all $T'>0$ to the Cauchy-Dirichlet boundary value problem $(\ref{equn_RF3-2})-(\ref{equn_ibp3-2})$, and $u$ converges locally uniformly to the Loewner-Nirenberg solution $u_{LN}$ in $C^2$ sense in $M$ as $t\to\infty$.
\end{thm}
In the following theorem (Theorem \ref{thm_asymptbdlowspeed1}) we show that when the Dirichlet boundary data $\phi$ increases slowly, then for the solution $u$ in Theorem \ref{thm_CDPconvergnl}, the geodesic curvature $k_{e^{2u}g}$ on $\partial M$ converges to $1$ uniformly as $t\to\infty$ and hence, $k_{e^{2u}g}$ is uniformly bounded.
\begin{thm}
Let $(\overline{M},g)$ be a compact surface with its interior $M$ and boundary $\partial M$. Under the condition of Theorem \ref{thm_CDPconvergnl}, we assume that there exists $T_3>0$ such that $\phi$ depends only on $t$, i.e., $\phi=\phi(t)$ for $t\geq T_3$, and $\phi'(t)\to0$ as $t\to\infty$. Then the solution $u(x,t)$ obtained in Theorem \ref{thm_CDPconvergnl}, which converges locally uniformly in $C^2$ to $u_{LN}$ in $M$, satisfies that the geodesic curvature on $\partial M$
\begin{align*}
k_{e^{2u}g}=e^{-u}(\frac{\partial u}{\partial n_g}+k_g)\to 1
\end{align*}
uniformly on $\partial M$ as $t\to\infty$.
\end{thm}
Notice that there is a big class of functions $\phi(t)$ such that $\phi'(t)\geq0$ and $\phi(t)\geq \log(\xi(t))$ for some low-speed increasing function $\xi(t)$  when $t\geq T$ for some $T>0$, and $\phi'(t)\to 0$ as $t\to\infty$. For instance, $\log(t+1),\,(t+1)^\alpha$ for $0<\alpha<1$, $\log(\log(t+100))$ and so on. We show in Theorem \ref{thm_lowspeedCDdataperturb} the same asymptotic behavior of the geodesic curvature for more general boundary data $\phi$. We also obtain that when $\phi$ increases fast to infinity as $t\to\infty$, the geodesic curvature goes to infinity in a certain speed, see Theorem \ref{thm_asymptbdfastspeed2}.

Then we use the comparison theorem and the a priori estimates on the asymptotic behavior of the geodesic curvatures of $\partial M$ for the Dirichlet boundary data $\phi$ of different growth ratios to obtain a priori estimates of the solution to the problem $(\ref{equn_RF3})-(\ref{equn_ibp3})$, and obtain a sufficient condition on the choice of a large class of $\psi$ for the long time existence and convergence of the flow. The following is the main theorem of the paper.
\begin{thm}\label{thm_convergCO3}
Let $(\overline{M},g)$ be a compact surface with its interior $M$ and boundary $\partial M$. Assume that $u_{01}\in C^{2+\alpha}(\overline{M})$ and $\phi_{01}\in C^{2+\alpha,1+\frac{\alpha}{2}}(\partial M\times[0,T])$ for all $T>0$ satisfy the condition in Theorem \ref{thm_CDPconvergnl}. Let $u_1$ be the solution to the Cauchy-Dirichlet problem $(\ref{equn_RF3-2})-(\ref{equn_ibp3-2})$ with the initial and boundary data $u_{01}$ and $\phi_{01}$. We assume that $u_0\in C^{2+\alpha}(\overline{M})$ and $\psi\in C^{1+\alpha,\frac{1}{2}+\frac{\alpha}{2}}(\partial M\times[0,T])$ for all $T>0$, and also the compatibility condition holds on $\partial M\times\{0\}$:
\begin{align}\label{equn_compatibility2-1}
\frac{\partial}{\partial n_g}u_0+k_{g}=\psi(\cdot,0) e^{u_0}.
\end{align}
Suppose $u_0$ and $\psi$ satisfy the following:
\begin{align}
&\label{equn_initialdatagc} u_0\geq u_{10}\,\,\,\,\text{on}\,\,\overline{M},\\
&\label{equn_boundarydatagc}k_{e^{2u_1}g}\leq \psi \leq y(t)^{\frac{1}{3}}-2,\,\,\,\text{on}\,\,\partial M\times[0,\infty).
\end{align}
where $y(t)\in C^3([0,\infty))$ is some positive function satisfying
\begin{align*}
y'\geq 3y+1
\end{align*}
for $t\in[0,\infty)$. Then there exists a unique solution $u\in C^{2+\alpha,1+\frac{\alpha}{2}}(M\times[0,T])$ to the initial-boundary value problem $(\ref{equn_RF3-1})-(\ref{equn_ibp3-1})$ for all $T>0$, and moreover, $u$ converges locally uniformly in $C^2$ to the Loewner-Nirenberg solution $u_{LN}$ on $M$.
\end{thm}
\begin{Remark}
For any given boundary data $\psi\in C^{1+\alpha,\frac{1}{2}+\frac{\alpha}{2}}(\partial M\times[0,T])$ for all $T>0$, and any function $\bar{u}_0\in C^{2,\alpha}(\overline{M})$ such that $\bar{u}_0>u_{10}$ on $\overline{M}$, let $\varepsilon=\inf_{\overline{M}}(\bar{u}_0-u_{10})$. Then for any $\epsilon\in(0,\varepsilon)$ and $\delta>0$ small, one can always find a function $u_0\in C^{2,\alpha}(\overline{M})$ such that $|u_0-\bar{u}_0|<\epsilon$ on $\overline{M}$, $u_0=\bar{u}_0$ at any point $x\in M$ with the distance $r(x)$ to $\partial M$ satisfying $r(x)\geq \delta$, and moreover, $\psi$ and $u_0$ satisfy the compatibility condition $(\ref{equn_compatibility2-1})$. In particular, $u_0$ satisfies $(\ref{equn_compatibility2-1})$ and $(\ref{equn_initialdatagc})$.
\end{Remark}
With the compatibility condition $(\ref{equn_compatibility2-1})$, the short time existence of the solution to $(\ref{equn_RF3-1})-(\ref{equn_ibp3-1})$ is standard. Indeed, this is a quasilinear parabolic equation with Robin boundary condition, and at $t=0$, and the equation is strictly parabolic. Then by the Inverse Function Theorem and standard methods from the parabolic equation \cite{LSU}\cite{Lieberman}, there exists $\epsilon>0$ such that there exists a unique solution $u\in C^{2+\alpha, 1+\frac{\alpha}{2}}(\overline{M}\times[0,\epsilon])$ to $(\ref{equn_RF3-1})-(\ref{equn_ibp3-1})$.

The control on the geodesic curvature of the Cauchy-Dirichlet problem of the normalized Ricci flow is an important part of this paper, see Section \ref{section_Asympbahavgeocurv1}. We now give an insight into the condition on the prescribed geodesic curvature $\psi$ on $\partial M$ for the initial-boundary value problem $(\ref{equn_RF3-1})-(\ref{equn_ibp3-1})$, in order that the solution $u$ converges to the Loewner-Nirenberg solution $u_{LN}$. Since the Cauchy-Dirichlet problem of the normalized Ricci flow is parallel to that of the Loewner-Nirenberg problem, we consider the following example: let $B_1(0)$ be the unit disk on $\mathbb{R}^2$. For any integer $m\geq 1$, to the boundary value problem
\begin{align}
&\Delta u_m=e^{2u_m},\,\,\text{in}\,\,B_1(0),\\
&u_m\big|_{\partial B_1(0)}=m,
\end{align}
there exists a unique solution $u_m=\log(\frac{2a_m}{1-a_m^2|x|^2})$ with $a_m^2=1+2e^{-2m}-2e^{-m}\sqrt{1+e^{-2m}}$, for $x\in B_1(0)\subseteq \mathbb{R}^2$. Therefore, the geodesic curvature on the boundary satisfies
\begin{align*}
k_{g_m}=e^{-u_m}(\frac{\partial}{\partial n_{g_0}}u_m+k_{g_0})=\frac{ 1-a_m^2|x|^2}{2a_m}[\frac{2a_m^2|x|}{1-a_m^2|x|^2}+k_{g_0}]\big|_{|x|=1}\to 1
\end{align*}
as $m\to +\infty$, which is exactly the limit of the geodesic curvature of the geodesic circles on the hyperbolic plane as the geodesic radius goes to infinity. In general, for the Loewner-Nirenberg problem on a compact manifold $(\overline{M},g)$ of dimension $n$ with boundary, the mean curvature of $\partial M$ of the sequence of solutions to the Yamabe equation with Dirichlet boundary data $u_m\big|_{\partial M}=m$ converges to $n-1$ as $m\to +\infty$, while a direct computation shows that the mean curvature of the level set of each geodesic defining function $x$ of the asymptotically hyperbolic manifold $(M, u_{LN}^{\frac{4}{n-2}}g)$ converges to $n-1$ as $x\to0$, where $u_{LN}$ is the solution to the Loewner-Nirenberg problem. So it seems reasonable to assume that $\psi\to 1$ as $t\to +\infty$. But on the other hand, we give two counter examples which show that even if $\psi\to 1$ as $t\to \infty$, the solution to $(\ref{equn_RF3-1})-(\ref{equn_ibp3-1})$ could diverge and hence does not converge to $u_{LN}$ locally in $M$, see Section \ref{section_counterexamples}, and the example at the end of Section \ref{section_comparisonthm} as a corollary of the comparison theorem--Theorem \ref{thm_comparison2}.

We notice that in this paper, for a two dimensional compact surface $(\overline{M},g)$ with boundary $\partial M$, the boundary $\partial M$ may have many connected components.

The paper is organized as follows: In Section \ref{section_counterexamples}, we give an example of the initial-boundary value problem $(\ref{equn_RF3-1})-(\ref{equn_ibp3-1})$ to which the solution $u$ diverges, although the boundary geodesic curvature $\psi$ is always close to $1$. In Section \ref{section_comparisonthm}, we provide an interior upper bound estimate, and give a comparison theorem of the problem $(\ref{equn_RF3-1})-(\ref{equn_ibp3-1})$ which plays an important role in the global a priori estimates of $u$ and the convergence of $u$, and we also give another example of the problem $(\ref{equn_RF3-1})-(\ref{equn_ibp3-1})$ with the boundary geodesic curvature $\psi\to 1$ as $t\to\infty$ and use the comparison theorem to show that the solution $u$ diverges. In Section \ref{section_convergCDP}, we prove that when the boundary data $\phi\to \infty$ in a speed not too slow as $t\to\infty$, the solution to the Cauchy-Dirichlet problem $(\ref{equn_RF3-2})-(\ref{equn_ibp3-2})$ converges locally uniformly in $M$ to the Loewner-Nirenberg solution $u_{LN}$ as $t\to\infty$. Then in Scetion \ref{section_Asympbahavgeocurv1}, we provide a careful analysis on the asymptotic behavior of the geodesic curvature on $\partial M$ for the solution of $(\ref{equn_RF3-2})-(\ref{equn_ibp3-2})$ as $t\to\infty$, for the cases when $\phi\to \infty$ in different speeds as $t\to\infty$. Finally, in Section \ref{section_convergenceibpgc}, we prove the main theorem: Indeed, we use the solutions to the auxiliary Cauchy-Dirichlet problem with certain low speed and high speed increasing boundary data to control the lower bound and upper bound of the solution $u$ to $(\ref{equn_RF3-1})-(\ref{equn_ibp3-1})$, by the comparison theorem. In Section \ref{section_furtherdiscussion}, we point out the similarity of initial-boundary value problem of the related curvature flows in higher dimensions with prescribed mean curvature on the boundary to the normalized $2$-D Ricci flow with prescribed geodesic curvature on the boundary; we also pose some open problems.

\section{An example of divergence of the flow}\label{section_counterexamples}

Now we take $\overline{M}$ to be the unit disk $\overline{B_1}$ on $\mathbb{R}^2$ with $B_1$ the interior, and define $g_0=e^{2u_0}g_E$ on $\overline{M}$ with $g_E$ the Euclidean metric. Hence by taking the Euclidean metric as the background metric, the system $(\ref{equn_RF3-1})-(\ref{equn_ibp3-1})$ can be written as
\begin{align}\label{equn_RF3-1-0}
&v_t=e^{-2v}\Delta v-1,\,\text{in}\,\overline{B_1}\times [0,\infty)\,\\
&\label{equn_bp3-1-0}\frac{\partial}{\partial r}v=\psi e^v-1,\,\,\text{on}\,\partial B_1 \times [0,\infty),\\
&\label{equn_ibp3-1-0}v\big|_{t=0}=u_0,\,\,\text{in}\,\overline{B_1},
\end{align}
where $\bar{g}=e^{2v}g_0$ and $r=|x|$ the Euclidean distance function to the origin. The compatibility condition $(\ref{equn_compatibility2-1})$ becomes
\begin{align}
&\psi(\cdot,0)=e^{-u_0}(\frac{\partial}{\partial r}u_0+1),
\end{align}
 at $r=1$. Let $[0,T)$ be the largest existence time interval of the solution $v$ to the initial-boundary value problem $(\ref{equn_RF3-1-0})-(\ref{equn_ibp3-1-0})$. Let $\xi(t)\equiv \displaystyle\sup_{x\in\overline{B_1}}v(x,t)$ for $T>t\geq0$. Then, $\xi=\xi(t)$ is locally Lipschitz in $[0,T)$, see \cite{Hamilton}. Denote $(\frac{d \xi}{dt})_+(t)=\displaystyle\limsup_{\tau\searrow 0}\frac{\xi(t+\tau)-\xi(t)}{\tau}$. By the maximum principle, for any $0<t_1<T$, $\displaystyle\sup_{\overline{B_1}\times [0,t_1]}v$ can only be achieved at points in $\overline{B_1}\times\{0\}\bigcup \partial B_1\times (0,t_1]$. If $v(x,t)=\displaystyle\sup_{\overline{B_1}\times [0,t_1]}v$ for some $(x,t)\in \partial B_1\times [0,t_1]$, then $\frac{\partial v}{\partial r}(x,t)\geq0$, and hence by $(\ref{equn_bp3-1-0})$, we have
 \begin{align}\label{inequn_bdmaxpoint}
 v(x,t)\geq -log(\psi(x,t)).
 \end{align}
Now for any small constant $\varepsilon_0\in(0,1)$, we choose the boundary data $\psi\leq 1+\varepsilon_0$ on $\partial B_1 \times[0,+\infty)$, and then choose $u_0$ on $\overline{B_1}$ so that
\begin{align*}
\displaystyle\sup_{\partial B_1}u_0< \displaystyle\sup_{ \overline{B_1}}u_0<-\log(1+\varepsilon_0),
\end{align*}
and the following holds
\begin{align}
&\psi(\cdot,0)=e^{-u_0}(\frac{\partial}{\partial r}u_0+1),
\end{align}
at $r=1$. (Just define $u_0$ near the interior maximum point and near the boundary first, and then extend $u_0$ smoothly on $\overline{B_1}$.) With these initial and boundary data, considering $(\ref{inequn_bdmaxpoint})$, by continuity, we have that $\xi(t)$ is achieved only by some interior points of $\overline{B_1}$ for $t>0 $ small. Let $t_1>0$ be the smallest time when $v(\cdot,t)$ achieves its maximum in $\overline{B_1}$ by a boundary point. For $t\in[0,t_1)$, we define the set
\begin{align*}
S(t)=\{x\in B_1 \big| v(x,t)=\displaystyle\sup_{\overline{B_1}}v(x,t)\}.
\end{align*}
Then $S(t)$ is a compact subset of $B_1$ for $t\in[0,t_1)$. By $(\ref{equn_RF3-1-0})$, at each $x\in S(t)$, we have $v_t(x,t)\leq-1$. Therefore, by Lemma 3.5 in \cite{Hamilton},
\begin{align*}
(\frac{d \xi}{dt})_+(t)\leq -1,
 \end{align*}
and hence $\xi(t)$ is non-increasing for $t\in(0,t_1)$. Let $x\in \partial B_1$ satisfy that $v(x,t_1)=\displaystyle\sup_{\overline{B_1}\times \{t_1\}}v$. Therefore, \begin{align*}
v(x,t_1)\leq\displaystyle\sup_{ \overline{B_1}}u_0<-\log(1+\varepsilon_0).
\end{align*}
But by $(\ref{equn_bp3-1-0})$, we have
 \begin{align}
 v(x,t_1)\geq -\log(\psi(x,t_1))\geq -\log(1+\varepsilon_0),
 \end{align}
which yields a contradiction. Therefore, for the initial data $u_0$ and boundary data $\psi$ we have just chosen, it is always true that for any $t\in(0,T)$, $\xi(t)$ is only achieved at some interior points $x$ of $\overline{B_1}$ and hence $S(t)$ is a compact subset of $B_1$. Then by $(\ref{equn_RF3-1-0})$, for $x\in S(t)$,
\begin{align}
v_t(x,t)\leq -1.
\end{align}
By Lemma 3.5 in \cite{Hamilton}, we have that $(\frac{d \xi}{dt})_+(t)\leq -1$, for $t\in[0,T)$, a.e. Therefore, for our choice of $\psi$ and $u_0$, even here $\psi$ could be chosen to be sufficiently close to $1$ for all time, the solution to the initial-boundary value problem $(\ref{equn_RF3-1-0})-(\ref{equn_ibp3-1-0})$ diverges and never converges to the Loewner-Nirenberg solution (which is corresponding to a complete hyperbolic metric).

\section{Interior upper bound estimates and a comparison theorem for $(\ref{equn_RF3-1})-(\ref{equn_ibp3-1})$}\label{section_comparisonthm}

For a compact surface $(\overline{M},g)$ with its interior $M$ and boundary $\partial M$, for any $\phi\in C^{4,\alpha}(\partial M)$, there exists a unique solution $u\in C^{4,\alpha}(\overline{M})$ to the following Dirichlet boundary value problem of the linear elliptic equation
\begin{align}
&\Delta_gu=K_g,\,\,\text{in}\,\,\overline{M},\\
&u=\phi,\,\,\text{on}\,\,\text{on}\,\,\partial M.
\end{align}
Then the conformal metric $h=e^{2u}g$ has zero Gauss curvature $K_h=0$. Since the initial-boundary problem $(\ref{equn_RF3-1})-(\ref{equn_ibp3-1})$ is conformally invariant, we will always choose a conformal metric $h$ such that $K_h=0$ as the background metric of the flow in this section, and hence $(\ref{equn_RF3-1})-(\ref{equn_ibp3-1})$ becomes
\begin{align}\label{equn_RF3-3}
&u_t=e^{-2u}\Delta_{h}u-1,\,\text{in}\,\overline{M}\times [0,\infty)\,\\
&\label{equn_bp3-3}\frac{\partial}{\partial n_{h}}u+k_{h}=\psi e^u,\,\,\text{on}\,\partial M \times [0,\infty),\\
&\label{equn_ibp3-3}u\big|_{t=0}=u_0,\,\,\text{in}\,\overline{M},
\end{align}
for some given initial data $u_0$, with the compatibility condition
\begin{align}
\frac{\partial}{\partial n_{h}}u_0+k_{h}=\psi(\cdot,0) e^{u_0},
\end{align}
on $\partial M$. We now discuss the uniform upper bounds on $u$ in any compact subset of $M$. We will prove the long time existence of the solution in Section \ref{section_convergenceibpgc}, after the comparison theorem (Theorem \ref{thm_comparison2}) and the estimates of geodesic curvature of solutions to the Cauchy-Dirichlet problem of the normalized Ricci flow in Section \ref{section_Asympbahavgeocurv1}. For any $\bar{x}\in M$, let $r$ be the distance of $\bar{x}$ to $\partial M$ in $(\overline{M},h)$. Let $\overline{U}=\overline{B_{s}}(\bar{x})$ be the closed $s$-geodesic ball of $\bar{x}$ and $U$ be its interior, with $s=\min\{\frac{r}{2}, \frac{r_{\text{inj}}(\bar{x})}{2}\}$ where $r_{\text{inj}}(\bar{x})$ is the injectivity radius at $\bar{x}$ in $(\overline{M},h)$. Let $u_{LN}$ be the Loewner-Nirenberg solution to the equation
\begin{align}
\Delta_hv=e^{2v}
\end{align}
in $U$, with the boundary condition
\begin{align}
\lim_{p\to \partial U}v(p)=+\infty.
\end{align}
 Let $[0,T)$ be the largest existence interval of the solution $u$ to $(\ref{equn_RF3-3})-(\ref{equn_ibp3-3})$. Let $\xi(t)=\displaystyle\sup_{x\in U}(u(x,t)-u_{LN}(x))$ for $0\leq t<T$. We define the set
 \begin{align*}
S(t)=\{x\in U \big| u(x,t)-u_{LN}(x)=\xi(t)\},
\end{align*}
for $t\geq0$. Then $S(t)$ is a compact subset in $U$ and $\xi(t)$ is Lipschitz on $[0,t_1]$ for any $0<t_1<T$. For any $x\in S(t)$, by $(\ref{equn_RF3-3})$, we have
\begin{align}
\frac{\partial}{\partial t}(u-u_{LN})&=[(e^{-2u}-e^{-2u_{LN}})\Delta_hu_{LN}+e^{-2u}\Delta_h(u-u_{LN})]\big|_{(x,t)}\\
&=[e^{2(u_{LN}-u)}-1+e^{-2u}\Delta_h(u-u_{LN})]\big|_{(x,t)}\\
&\leq [e^{2(u_{LN}-u)}-1]\big|_{(x,t)}\\
&=e^{-2\xi(t)}-1,
\end{align}
for $t\geq0$. By the maximum principle of the equation satisfied by $u-u_{LN}$, if $\xi(t_1)\leq 0$ for some $t_1\geq0$, then $\xi(t)\leq0$ for all $t\geq t_1$. Denote $(\frac{d \xi}{dt})_+(t)=\displaystyle\limsup_{\tau\searrow 0}\frac{\xi(t+\tau)-\xi(t)}{\tau}$. By Lemma 3.5 in \cite{Hamilton},
 \begin{align}
(\frac{d \xi}{dt})_+(t)\leq \sup\{\frac{\partial}{\partial t}(u-u_{LN})(x,t)\big|\,x\in S(t)\}\leq e^{-2\xi(t)}-1,
\end{align}
for $t\geq0$. Therefore, $\xi(t)$ is uniformly bounded from the above on the interval $[0,T)$. Moreover, if $T=+\infty$, then $\limsup_{t\to\infty}\xi(t)\leq 0$. Therefore, we have a uniform upper bound of $u$ on each compact subset $K$ of $M$ for $t\in[0,T)$. The a priori estimates near the boundary are difficult when the boundary data $\psi>0$.

Now we prove a comparison theorem.
\begin{thm}\label{thm_comparison2}
Let $K_h=0$ on $\overline{M}$. For $i=1,2$, let $u_i$ be the unique solution to the boundary value problem
\begin{align}\label{equn_RF3-3i}
&u_t=e^{-2u}\Delta_{h}u-1,\,\text{in}\,\overline{M}\times [0,T),\,\\
&\label{equn_bp3-3i}\frac{\partial}{\partial n_{h}}u+k_{h}=\psi_i e^u,\,\,\text{on}\,\partial M \times [0,T),\\
&\label{equn_ibp3-3i}u\big|_{t=0}=u_{0i},\,\,\text{in}\,\overline{M},
\end{align}
with $u_{01}\leq u_{02}$ on $\overline{M}$ and $\psi_1\leq \psi_2$  on $\partial M \times [0,T)$. Then we have
\begin{align}
u_1\leq u_2
\end{align}
for $(x,t)\in \overline{M}\times [0,T)$.
\end{thm}

\begin{proof}
Let $v=u_1-u_2$. For any $t_1\in (0,T)$, we will show that $u_1\leq u_2$ on $\overline{M}\times [0,t_1]$. By $(\ref{equn_RF3-3i})$, we have
\begin{align*}
e^{2u_i}(u_i)_t+e^{2u_i}=\Delta_{h}u_i,\,\text{in}\,M\times [0,t_1]
\end{align*}
for $i=1,2$. By taking difference between these two equations, we have
\begin{align*}
(e^{2u_1}-e^{2u_2})(\frac{\partial u_1}{\partial t}+1)+e^{2u_2}v_t=\Delta_{h}v,\,\text{in}\,M\times [0,t_1],
\end{align*}
and hence,
\begin{align}\label{equn_vrev}
v_t=e^{-2u_2}\Delta_hv-f v, \,\text{in}\,M\times [0,t_1],
\end{align}
where $f=(\frac{\partial u_1}{\partial t}+1)\times\frac{(e^{2v}-1)}{v}$ is a continuous function on $\overline{M}\times [0,t_1]$. Similarly, by $(\ref{equn_bp3-3i})$, we have
\begin{align}\label{equn_bddv}
\frac{\partial v}{\partial n_h}&=(e^{u_1}-e^{u_2})\psi_1+e^{u_2}(\psi_1-\psi_2)\\
&=\tilde{f}\,v+e^{u_2}(\psi_1-\psi_2),\notag
\end{align}
in $\partial M\times [0,t_1]$, where $\tilde{f}=\frac{(e^{u_1}-e^{u_2})\psi_1}{u_1-u_2}$ is continuous on $\partial M\times [0,t_1]$ and it is positive when $\psi_1>0$. Let $r$ be the distance function to the boundary $\partial M$ on $(\overline{M},h)$. There exits a small constant $\frac{1}{10}>\epsilon>0$ such that $r$ is smooth in the $\epsilon$-neighborhood $\overline{U_{\epsilon}}$ of $\partial M$, and the exponential map $F:\,\partial M\times[0,\epsilon]\to \overline{U_{\epsilon}}$ is a diffeomorphism, with $F(p,r)=\text{Exp}_p(-r n)$ where $n$ is the unit outer norm of $\partial M$ at $p$. We define the function $\tau(x)\in C^{2,\alpha}(\overline{M})$ such that $\tau(x)\geq \frac{1}{2}$ on $\overline{M}$, and $\tau(x)=1-r(x)$ for $0\leq r(x)\leq \epsilon$. Now let $\xi=e^{-at}v\eta(x)$, with $a>0$ a large constant to be chosen, and $\eta(x)=e^{-N\tau(x)}$ for some constant $N>0$ to be defined. By $(\ref{equn_vrev})$, we have
\begin{align}
\xi_t&=e^{-2u_2}e^{-at}\eta\Delta_hv+(-f -a)\xi\\
&=e^{-2u_2}\Delta_h\xi-2e^{-2u_2}e^{-at}\nabla_h\eta\cdot\nabla_hv-e^{-2u_2}e^{-at}v\Delta_h\eta+(-f -a)\xi \notag\\
&=e^{-2u_2}\Delta_h\xi-2e^{-2u_2}\nabla_h\eta\cdot\nabla_h(v e^{-at}\eta) \eta^{-1}+2e^{-2u_2}\eta^{-1}(\nabla_h\eta\cdot\nabla_h\eta)\xi-e^{-2u_2}\eta^{-1}\Delta_h\eta\,\xi+(-f -a)\xi\notag\\
&=e^{-2u_2}\Delta_h\xi+X\cdot \nabla_h\xi+[\zeta-f-a]\xi,\notag
\end{align}
where $X=-2e^{-2u_2}\eta^{-1}\nabla_h\eta$ is a continuous vector field on $\overline{M}\times [0,t_1]$ depending on $N$, and $\zeta=2e^{-2u_2}\eta^{-1}\big|\nabla_h\eta\big|_h^2-e^{-2u_2}\eta^{-1}\Delta_h\eta$ is a continuous function on $\overline{M}\times [0,t_1]$ depending on $N$. Then for a fixed large number $N>0$, we choose $a>0$ sufficiently large so that
\begin{align*}
\zeta-f-a<0
\end{align*}
on $\overline{M}\times [0,t_1]$. For the choice of $N>0$, by $(\ref{equn_bddv})$, we have
\begin{align}
\frac{\partial \xi}{\partial n_h}&=e^{-at}\eta\tilde{f}\,v+e^{u_2}e^{-at}\eta(\psi_1-\psi_2)-e^{-at}v\frac{\partial\eta}{\partial r}\\
&=(\tilde{f}-N)\xi+e^{u_2}e^{-at}\eta(\psi_1-\psi_2),
\end{align}
on $\partial M \times [0,t_1]$, where $\tilde{f}$ is a continuous positive function on $\partial M \times [0,t_1]$ depending on $\psi_1$. We now choose $N>0$ large so that $\tilde{f}-N<0$ on $\partial M \times [0,t_1]$. Recall that $\xi(x,0)\leq 0$ for $x\in \overline{M}$, and $\psi_1-\psi_2\leq 0$ on $\partial M\times [0,t_1]$. Using the maximum principle on $\xi$ in the domain $\overline{M}\times [0,t_1]$, we have that $\xi\leq 0$ on $\overline{M}\times [0,t_1]$ for any $t_1\in(0,T)$, and hence, $u_1\leq u_2$ on $\overline{M}\times [0,T)$.

\end{proof}
\begin{Remark}\label{remark_comparison}
Using the same proof of Theorem \ref{thm_comparison2}, we have the comparison theorem still holds when $u_1$ is a subsolution to $(\ref{equn_RF3-3i})$ satisfying
\begin{align*}
&u_t\leq e^{-2u}\Delta_{h}u-1,\,\text{in}\,\overline{M}\times [0,T),\,\\
&\frac{\partial}{\partial n_{h}}u+k_{h}\leq\psi_1 e^u,\,\,\text{on}\,\partial M \times [0,T),\\
&u\big|_{t=0}\leq u_{01},\,\,\text{in}\,\overline{M},
\end{align*}
and $u_2$ is a supersolution to $(\ref{equn_RF3-3i})$ satisfying
\begin{align*}
&u_t\geq e^{-2u}\Delta_{h}u-1,\,\text{in}\,\overline{M}\times [0,T),\,\\
&\frac{\partial}{\partial n_{h}}u+k_{h}\geq\psi_2 e^u,\,\,\text{on}\,\partial M \times [0,T),\\
&u\big|_{t=0}\geq u_{02},\,\,\text{in}\,\overline{M},
\end{align*}
with $u_{01}\leq u_{02}$ on $\overline{M}$ and $\psi_1\leq \psi_2$ on $\partial M\times[0,\infty)$.
\end{Remark}
This comparison theorem will be important for the discussion of the choice of the initial data $u_0$ and the boundary data $\psi$, and the long time existence and the convergence of the corresponding solutions.

As a consequence of Theorem \ref{thm_comparison2}, we now show that the solution to the following boundary value problem will not converge to the Loewner-Nirenberg metric. Let $(\overline{M},h)$ be the Euclidean unit disk $(\overline{B_1},g_E)\subseteq \mathbb{R}^2$ and $r$ the distance to the origin. We consider the initial-boundary value problem $(\ref{equn_RF3-3})-(\ref{equn_ibp3-3})$ and let $u$ be the unique solution, where $u_0$ is a radial symmetric solution to the equation
\begin{align}
\Delta u_0=e^{2u_0}
\end{align}
in $\overline{B_1}$, and $\psi=\psi(t)$ satisfies $\frac{\partial u_0}{\partial r}=\psi e^{u_0}-1$ on $\partial B_1 \times \{0\}$. For instance, just take $u_0=\log(\frac{2b}{1-b^2|x|^2})$ with some constant $b\in (0,1)$. Then we have $\psi(0)>1$. We then choose $\psi=\psi(t)$ so that $\psi'(t)\leq0$ for $t\geq0$ and $\displaystyle\lim_{t\to+\infty}\psi(t)=1$. Recall that $u_0$ is the unique solution to the initial-boundary value problem
\begin{align*}
&u_t=e^{-2u}\Delta_{h}u-1,\,\text{in}\,\overline{B_1}\times [0,\infty)\,\\
&\frac{\partial}{\partial n_{h}}u+k_{h}=\psi(0) e^u,\,\,\text{on}\,\partial B_1 \times [0,\infty),\\
&u\big|_{t=0}=u_0,\,\,\text{in}\,\overline{B_1},
\end{align*}
Since $\psi(t)\leq \psi(0)$ for $t\geq0$, by Theorem \ref{thm_comparison2}, we have that $u(x,t)\leq u_0(x)$ for any $(x,t)\in \overline{B_1}\times [0,\infty)$. Hence, the solution $u$ will not converge to the Loewner-Nirenberg solution as $t\to\infty$ on any compact subset of $B_1$.

\section{Convergence of the Cauchy-Dirichlet problem  of the normalized Ricci flow}\label{section_convergCDP}

For a compact surface $(\overline{M},h)$ with its interior $M$ and boundary $\partial M$, without loss of generality, we take $h$ to be the background conformal metric such that the Gauss curvature $K_h=0$. Then $(\ref{equn_RF3-2})-(\ref{equn_ibp3-2})$ can be written as
 \begin{align}\label{equn_RF3-2-1}
&u_t=e^{-2u}\Delta_{h}u-1,\,\text{in}\,\overline{M}\times [0,\infty)\,\\
&u=\phi(\cdot,t),\,\,\text{on}\,\partial M \times [0,\infty),\\
&\label{equn_ibp3-2-1}u\big|_{t=0}=u_0,\,\,\text{in}\,M,
 \end{align}
with the compatibility condition
 \begin{align}\label{equn_RF3-2-1cpr1}
&u_0(x)=\phi(x,0),\,\,\text{for}\,\,x\in \partial M,\\
&\label{equn_RF3-2-1cpr2}\frac{\partial \phi}{\partial t}(x,0)=e^{-2u_0(x)}\Delta_{h}u_0(x)-1,\,\,\text{for}\,\,x\in\,\partial M.
 \end{align}
In order that $u\in C^{4+\alpha,2+\frac{\alpha}{2}}(\overline{M}\times[0,\epsilon])$ for some $\epsilon>0$, we need the additional condition
\begin{align}\label{equn_RF3-2-1cpr42}
\frac{\partial^2\phi}{\partial t^2}(x,0)=e^{-2u_0}[-2\Delta_hu_0(e^{-2u_0}\Delta_{h}u_0-1)+\Delta_h(e^{-2u_0}\Delta_{h}u_0)](x),
\end{align}
for $x\in \partial M$. Notice that when $u_0$ is a solution to equation
\begin{align}\label{equn_elliptYb}
\Delta_{h}u=e^{2u},
\end{align}
on $\overline{M}$, then by $(\ref{equn_RF3-2-1cpr2})$ and $(\ref{equn_RF3-2-1cpr42})$, we have $\phi_t(x,0)=\phi_{tt}(x,0)=0$ for $x\in \partial M$. The strategy in this section is parallel to those in \cite{GLi} and \cite{GLi2}.

\begin{defn}
We call a function $\xi(t)\in C^1([0,+\infty))$ a low-speed increasing function if, $\xi(t)>0$ for $t\geq0$, $\displaystyle\lim_{t\to\infty}\xi(t)=\infty$, and there exist two constants $T>0$ and $1>\tau>0$ such that for $t\geq T$, we have
\begin{align}\label{inequn_lowspeedinc}
\xi'(t)\leq \tau.
\end{align}
\end{defn}

We first show the long time existence of the solution to the Cauchy-Dirichlet problem $(\ref{equn_RF3-2-1})-(\ref{equn_ibp3-2-1})$.
\begin{thm}\label{thm_longexistence1}
Let $(\overline{M},h)$ be a compact surface with its interior $M$ and boundary $\partial M$ such that $K_h=0$. Assume that $u_0\in C^{2,\alpha}(\overline{M})$ and $\phi\in C^{2+\alpha,1+\frac{\alpha}{2}}(\partial M\times[0,T])$ for all $T>0$ such that the compatibility condition $(\ref{equn_RF3-2-1cpr1})-(\ref{equn_RF3-2-1cpr2})$ holds. Then there exists a unique solution $u$ to the Cauchy-Dirichlet problem $(\ref{equn_RF3-2-1})-(\ref{equn_ibp3-2-1})$ for all time $t\geq0$, and $u\in C^{2+\alpha,1+\frac{\alpha}{2}}(\overline{M}\times[0,T])$ for all $T>0$. Moreover, if $u_0\in C^{4+\alpha}(\overline{M})$ and $\phi\in C^{4+\alpha,2+\frac{\alpha}{2}}(\partial M\times[0,T])$ for $T>0$, and also the compatibility condition $(\ref{equn_RF3-2-1cpr42})$ holds, then $u\in C^{4+\alpha,2+\frac{\alpha}{2}}(\overline{M}\times[0,T])$ for all $T>0$.
\end{thm}
\begin{proof}
The short time existence of the solution is classical. Indeed, at $t=0$, the equation is a strictly parabolic equation. By the compatibility condition, using the inverse function theorem and standard methods from the parabolic equation \cite{LSU}, there exists $\epsilon>0$ such that there exists a unique solution $u\in C^{2+\alpha,1+\frac{\alpha}{2}}(\overline{M}\times[0,\epsilon])$ to the Cauchy-Dirichlet problem. By the standard methods of the parabolic equation, to obtain the long time existence, we only need to show the upper bound and lower bound estimates on $u$ on $\overline{M}\times[0,T]$ for any $T>0$.

By the maximum principle, $u(x,t)\leq \max\{\sup_{\overline{M}}u_0,\,\sup_{\partial M\times[0,T]}\phi\}$ for $(x,t)\in \overline{M}\times[0,T]$ for any $T>0$.

For the lower bound estimates, let $[0,T)$ be the largest existence time interval of the solution $u$. Let $\xi(t)=\inf_{\overline{M}}u(\cdot,t)$. Then $\xi(t)$ is Lipschitz in $[0,t_1]$ for any $t_1\in(0,T)$. Therefore, the derivative $\xi'(t)$ exists for almost every $t$ and $\xi(t)-\xi(0)=\int_0^t\xi'(s)ds$ for $t\in(0,T)$. We define $(\frac{d \xi}{dt})_+(t)=\displaystyle\limsup_{\tau\searrow 0}\frac{\xi(t+\tau)-\xi(t)}{\tau}$. Denote the set
\begin{align*}
S(t)=\{x\in\overline{ M}\big|\,u(x,t)=\inf_{\overline{M}}u(\cdot,t)\},
\end{align*}
for $t\geq0$. Then $S(t)$ is compact. By Lemma 3.5 in \cite{Hamilton},
 \begin{align}
(\frac{d \xi}{dt})_+(t)\geq \inf\{\frac{\partial}{\partial t}u(x,t)\big|\,x\in S(t)\}\geq \min\{\inf_{\partial M}\frac{\partial\phi}{\partial t}(\cdot,t),\,-1\},
\end{align}
for $t\geq0$ and hence, $\xi(t)$ is uniformly bounded from below in $[0,T)$ if $T<+\infty$. Therefore, we obtain the uniform upper and lower bounds on $u$ for any finite time intervals. By the standard theory of parabolic equations, the $C^{2+\alpha,1+\frac{\alpha}{2}}(\overline{M}\times[0,t_1])$ norm of $u$ is uniformly controlled for any $t_1>0$ and hence, the solution $u$ exists for all $t\geq0$. Also, under the compatibility condition $(\ref{equn_RF3-2-1cpr42})$, by the standard regularity argument of parabolic equations, we obtain the required regularity in the theorem. This completes the proof of the theorem.

\end{proof}

Now we give a comparison lemma.

\begin{lem}\label{lem_comparisonCDP1}
Let $(\overline{M}^2,h)$ be a compact surface with its interior $M$ and boundary $\partial M$. Assume the Gauss curvature $K_h=0$. Let $u$ be a sub-solution to $(\ref{equn_RF3-2-1})$, i.e., $u_t\leq e^{-2u}\Delta_{h}u-1,\,\text{in}\,\overline{M}\times [0,T)$, and $v$ be a sup-solution to $(\ref{equn_RF3-2-1})$, i.e., $v_t\geq e^{-2v}\Delta_{h}v-1,\,\text{in}\,\overline{M}\times [0,T)$. Suppose that $u\leq v$ on $\partial M\times[0,T) \bigcup \overline{M}\times \{0\}$. Then we have $u\leq v$ on $\overline{M}\times [0,T)$.
\end{lem}
\begin{proof}
Let $\xi=u-v$. By the condition in the lemma, we have
\begin{align}\label{inequn_comparisonCD}
\xi_t\leq e^{-2u}\Delta_h\xi+f(x,t)\xi,
\end{align}
in $\overline{M}\times [0,T)$, where $f(x,t)=-2(1+v_t)$ when $\xi=0$; while $f(x,t)=(1+v_t)\xi^{-1}(e^{-2\xi}-1)$ when $\xi\neq0$. It is easy to check that $f\in C(\overline{M}\times [0,T))$. Since $\xi\leq 0$ on $\partial M\times[0,T) \bigcup \overline{M}\times \{0\}$, by the classical maximum principle of the parabolic inequality $(\ref{inequn_comparisonCD})$, we have $\xi\leq 0$ on $\overline{M}\times [0,T)$ and hence, $u\leq v$ on $\overline{M}\times [0,T)$. This proves the lemma.

\end{proof}

For the convergence of the flow, we first consider a simple case.

\begin{thm}\label{thm_convertibp1-a}
Let $(\overline{M},h)$ be a compact surface with its interior $M$ and boundary $\partial M$ such that $K_h=0$. Assume that $u_0\in C^{4+\alpha}(\overline{M})$ and $\phi\in C^{4+\alpha,2+\frac{\alpha}{2}}(\partial M\times[0,T])$ for all $T>0$, and also the compatibility condition $(\ref{equn_RF3-2-1cpr1})-(\ref{equn_RF3-2-1cpr42})$ holds. Moreover, let $u_0$ be a subsolution to $(\ref{equn_elliptYb})$ such that
\begin{align}\label{inequn_increasingcond1}
e^{-2u_0}[-2\Delta_hu_0(e^{-2u_0}\Delta_{h}u_0-1)+\Delta_h(e^{-2u_0}\Delta_{h}u_0)](x)\geq0,
\end{align}
at the points $x\in \partial M$ where it holds that $e^{-2u_0}\Delta_{h}u_0-1=0$. If $\phi_t\geq 0$ on $\partial M\times[0,+\infty)$, and $\phi(x,t)\geq \log(\xi(t))$ for $(x,t)\in \partial M\times[T_1,+\infty)$ with some constant $T_1>0$, where $\xi$ is a low-speed increasing function satisfying $(\ref{inequn_lowspeedinc})$, then we have that there exists a unique solution $u\in C^{4+\alpha,2+\frac{\alpha}{2}}(\overline{M}\times[0,T])$ for all $T>0$ to the Cauchy-Dirichlet problem $(\ref{equn_RF3-2-1})-(\ref{equn_ibp3-2-1})$ for all $t\geq0$, and moreover, $u$ converges locally uniformly to the Loewner-Nirenberg solution $u_{LN}$ in $C^4$ sense in $M$.
\end{thm}
We remark that the condition $(\ref{inequn_increasingcond1})$ is to guarantee the compatibility condition $(\ref{equn_RF3-2-1cpr42})$ and $\phi_t\geq0$ on $\partial M\times[0,T])$ to hold. When $u_0$ is a solution to $(\ref{equn_elliptYb})$ in a neighborhood of $\partial M$, then $(\ref{inequn_increasingcond1})$ holds automatically; while when $u_0$ is a strict subsolution to $(\ref{equn_elliptYb})$, then the condition $(\ref{inequn_increasingcond1})$ is not needed anymore. For instance, if $u_0$ is a subsolution to $(\ref{equn_elliptYb})$, then $u_0-C$ is a strictly subsolution for any constant $C>0$.

\begin{proof}
The long time existence of the solution $u$ is proved in Theorem \ref{thm_longexistence1}. So we only need to show the convergence.

We first use the maximum principle to show that $u_t\geq0$ on $\overline{M}\times[0,\infty)$. Let $v=u_t$. By $(\ref{equn_RF3-2-1})-(\ref{equn_ibp3-2-1})$, we have that
\begin{align}
&v_t=e^{-2u}\Delta_hv-2(v^2+v),\,\,\text{in}\,\overline{M}\times [0,\infty),\\
&v=\phi_t\geq0,\,\,\text{on}\,\,\partial M\times[0,\infty),\\
&v(x,0)=e^{-2u_0(x)}\Delta_{h}u_0(x)-1\geq0,\,\,\text{for}\,\,x\in \overline{M}.
\end{align}
By the maximum principle, $v$ can never be negative in $\overline{M}\times [0,\infty)$. Therefore, $u_t\geq0$ in $\overline{M}\times [0,\infty)$.

For each smooth compact domain $D \subseteq M$ with its interior $D^\circ$, let $u_{LN,D}$ be the Loewner-Nirenberg solution to $(\ref{equn_elliptYb})$ on $D$. By comparison, we have $u_0\leq u_{LN,D}$ in $D$. Now $\zeta=u-u_{LN,D}$. We have
\begin{align*}
&\zeta_t=e^{-2u}\Delta_h\zeta+(e^{-2\zeta}-1),\,\,\text{in}\,\,D^\circ\times[0,\infty),\\
&\zeta=-\infty,\,\,\text{on}\,\,\partial D\times[0,\infty),\\
&\zeta(x,0)\leq0\,\,\text{for}\,\,x\in\,D^\circ.
\end{align*}
Then, by the maximum principle, we have $\zeta\leq0$ in $D^\circ\times[0,\infty)$. Therefore, on each compact subset in $M$, $u$ is uniformly bounded from above and moreover $u$ is increasing in $t$ for any $x\in M$. By the standard interior estimates of the parabolic equation, we have that on each compact subset $K$ in $M$, $\|u\|_{C^{4+\alpha}}(K)$ is uniformly bounded from above, independent of $t\geq0$. Therefore, $u(x,t)\to u_{\infty}(x)$ locally uniformly in $M$ in $C^4$ sense. Also, by Harnack's inequality of $v$, we have $v(x,t)\to0$ locally uniformly on $M$ as $t\to\infty$ and hence, $u_{\infty}$ is a solution to $(\ref{equn_elliptYb})$ in $M$. Similarly, by comparison, $u(x,t)\leq u_{LN}(x)$ for $x\in M$, where $u_{LN}$ is the Loewner-Nirenberg solution to $(\ref{equn_elliptYb})$ on $M$. For the lower bound estimates of $u$ near $\partial M$, we need the following lemma, which is a modification of Lemma 4.4 in \cite{GLi2}.

\begin{lem}\label{lem_lowerboundbdary}
Let the condition of Theorem \ref{thm_convertibp1-a} holds. Let $r(x)$ be the distance of $x\in M$ to $\partial M$ in $(\overline{M},h)$. Then for any $\varepsilon>0$, there exists $\delta_1>0$ small and $T_2>T_1$, such that
\begin{align}
u(x,t)\geq -\log(r(x)+\epsilon(t))+w(x)-\varepsilon,
\end{align}
for $t\geq T_2$ and any $x\in M$ such that $r(x)\leq \delta_1$, where $\epsilon=\xi(t)^{-1}$, and $w$ is of $C^2$ in the $\delta_1$-neighborhood of $\partial M$ in $(\overline{M},h)$ such that $w(x)\leq0$, with $w=0$ on $\partial M$.
\end{lem}
\begin{proof}
For $q\in \partial M$ and $\delta_1>0$, we denote Exp$_q(r)$ the point on the geodesic starting from $q$ in the direction of the inner normal vector with distance $r$ to $q$ in $(\overline{M},h)$. Define the map $F:\, \partial M\times [0,\delta_1]\to \overline{M}$ such that $F(q,r)=$ Exp$_q(r)$. Then for $\delta_1>0$ sufficiently small, $F$ is a diffeomorphism to the image $U_{\delta_1}=\,\{\text{Exp}_q(r)\big|(q,r)\in\partial M\times [0,\delta_1]\}$. Then on $U_{\delta_1}$ the metric has the orthogonal decomposition
\begin{align*}
h=dr^2+h_r=dr^2+f(r,s)ds^2,
\end{align*}
with $h_r$ the restriction of $h$ on $\Sigma_r=\{p\in M\big| r(p)=r\}$. Now we define the function
\begin{align*}
\underline{u}(x,t)=-\log(r(x)+\epsilon(t))+w(x)-\varepsilon,
\end{align*}
for $(x,t)\in U_{\delta_1}\times[T,+\infty)$, where
\begin{align*}
w(x)=A(\frac{1}{(r(x)+\delta)^p}-\frac{1}{\delta^p})\leq0,
\end{align*}
with constants $A>0$, $p>1$ large and $\delta>0$ small to be defined. Recall that
\begin{align}
-\frac{\epsilon'(t)}{\epsilon(t)^2}=\xi'(t)\leq \tau,
\end{align}
for $t\geq T$. Let $\tilde{r}(x,t)=r(x)+\epsilon(t)$. Then we have
\begin{align*}
\underline{u}_t&=\frac{-\epsilon'(t)}{\tilde{r}(x,t)}\leq \frac{\tau \epsilon(t)^2}{\tilde{r}(x,t)}\leq \tau \epsilon(t),\\
e^{-2\underline{u}}\Delta_h \underline{u}-1&=e^{2\varepsilon-2w(x)} \tilde{r}(x,t)^2[\frac{\partial^2}{\partial r^2}\underline{u}+\frac{1}{2}f^{-1}\frac{\partial f}{\partial r}\frac{\partial}{\partial r}\underline{u}]-1\\
&=e^{2\varepsilon-2w(x)} \tilde{r}(x,t)^2[\frac{1}{\tilde{r}(x,t)^2}+Ap(p+1)(r(x)+\delta)^{-p-2}-\frac{1}{2}f^{-1}\frac{\partial f}{\partial r}\tilde{r}(x,t)^{-1}\\
&\,\,\,\,\,\,\,\,\,\,\,\,\,\,\,\,\,\,\,\,\,\,\,\,\,\,\,\,\,\,\,\,\,\,\,\,\,\,\,\,\,\,-\frac{Ap}{2}f^{-1}\frac{\partial f}{\partial r}(r(x)+\delta)^{-p-1}]-1\\
&=e^{2\varepsilon-2w(x)}[1+\frac{Ap(p+1)\tilde{r}(x,t)^2}{(r(x)+\delta)^{p+2}}-\frac{1}{2}f^{-1}\frac{\partial f}{\partial r}\tilde{r}(x,t)\\
&\,\,\,\,\,\,\,\,\,\,\,\,\,\,\,\,\,\,\,\,\,\,\,\,\,\,\,\,\,\,\,\,\,\,\,\,\,\,\,\,\,\,-\frac{Ap}{2}f^{-1}\frac{\partial f}{\partial r}(r(x)+\delta)^{-p-1}\tilde{r}(x,t)^2]-1\\
&=e^{2\varepsilon-2w(x)}[1+\frac{Ap\tilde{r}(x,t)^2}{(r(x)+\delta)^{p+2}}(p+1-\frac{(r(x)+\delta)\frac{\partial f}{\partial r}}{2f})-\frac{1}{2}f^{-1}\frac{\partial f}{\partial r}\tilde{r}(x,t)]-1,
\end{align*}
Now we fix $\delta>0$ small. Therefore, for $\delta_1>0$ sufficiently small,  $T_2>\max\{T_1,T\}$ sufficiently large and $p>0$ large, we have that
\begin{align*}
e^{-2\underline{u}}\Delta_h \underline{u}-1\geq e^{2\varepsilon-2w(x)}[1-\frac{\varepsilon}{2}]-1>\varepsilon>\underline{u}_t,
\end{align*}
and hence, $\underline{u}$ is a subsolution of the equation $(\ref{equn_RF3-2-1})$ for $(x,t)\in U_{\delta_1}\times[T_2,\infty)$. By the boundary condition, we have $u\geq \underline{u}$ on $\partial M \times [T_2,\infty)$. Notice that $\partial U_{\delta_1}=\partial M \bigcup \Sigma_{\delta_1}$. Since $u(x,t)$ is increasing in $t$, we have that $u(x,t)\geq u_0(x)$. On the other hand, there exists $A_1>0$ such that for each $A\geq A_1$ and $p>1$, we have that
\begin{align*}
-\log(\delta_1)+A[(\delta_1+\delta)^{-p}-\delta^{-p}]<\displaystyle\inf_{\Sigma_{\delta_1}}u_0,
\end{align*}
and hence we have on $\Sigma_{\delta_1}\times[0,\infty)$,
\begin{align*}
\underline{u}\leq u.
\end{align*}
On $U_{\delta_1}\times\{T_2\}$, since $\underline{u}(\cdot,T_2),\,u(\cdot,T_2)\in C^1(U_{\delta_1})$ and $\underline{u}\leq u$ on $\partial M\times\{T_2\}$, there exist $A_2>0$ and $p_2>0$ such that for $A\geq A_2$ and $p\geq p_2$, we have that on $U_{\delta_1}\times\{T_2\}$,
\begin{align*}
\underline{u}\leq u.
\end{align*}
Now we apply Lemma \ref{lem_comparisonCDP1} on $U_{\delta_1}\times[T_2,\infty)$ and obtain that
\begin{align*}
u\geq \underline{u},
\end{align*}
on $U_{\delta_1}\times[T_2,\infty)$, for some $A>1$ and $p>1$ large. This completes the proof of Lemma \ref{lem_lowerboundbdary}.

\end{proof}

 By Lemma \ref{lem_lowerboundbdary}, we have that
 \begin{align}
 u_{\infty}(x)\geq -\log(r(x))+w(x)-\varepsilon,
 \end{align}
 in a neighborhood of $\partial M$ relating to $\varepsilon$ and hence, $u_{\infty}$ is the Loewner-Nirenberg solution to $(\ref{equn_elliptYb})$.

\end{proof}
Now we consider the convergence of the flow for more general initial and boundary data.
\begin{thm}\label{thm_Flow1convergence}
Let $(\overline{M},h)$ be a compact surface with its interior $M$ and boundary $\partial M$ such that $K_h=0$. Assume that $u_0\in C^{2+\alpha}(\overline{M})$ and $\phi\in C^{2+\alpha,1+\frac{\alpha}{2}}(\partial M\times[0,\tau])$ for all $\tau>0$, and also the compatibility condition $(\ref{equn_RF3-2-1cpr1})-(\ref{equn_RF3-2-1cpr2})$ holds. Moreover, we assume that $\phi_t(x,t)\geq0$ on $\partial M\times[T,\infty)$ for some $T>0$, and $\phi(x,t)\geq \log(\xi(t))$ for $(x,t)\in \partial M\times[T_1,+\infty)$ with some constant $T_1>0$ where $\xi$ is a low-speed increasing function satisfying $(\ref{inequn_lowspeedinc})$. Then there exists a unique solution $u\in C^{2+\alpha,1+\frac{\alpha}{2}}(\overline{M}\times[0,T])$ for all $T>0$ to the Cauchy-Dirichlet boundary value problem $(\ref{equn_RF3-2-1})-(\ref{equn_ibp3-2-1})$ for all $t\geq0$, and $u$ converges locally uniformly to the Loewner-Nirenberg solution $u_{LN}$ in $C^2$ sense in $M$.
\end{thm}

\begin{proof}
The long time existence of the solution $u$ is proved in Theorem \ref{thm_longexistence1}. So we only need to show the convergence.

Let $\beta=\displaystyle\inf_{M\times[0,T]}u$. Let $v_0$ be the solution to $(\ref{equn_elliptYb})$ with the boundary condition $v_0\big|_{\partial M}=\beta-3$. Then by the maximum principle, $v_0\leq \beta-3$ on $\overline{M}$. We then pick up a function $\phi_0(x,t)\in C^{4+\alpha,2+\frac{\alpha}{2}}(\overline{M}\times[0,t_1])$ for all $t_1>0$ such that $v_0$ and $\phi_0$ satisfy the compatibility condition $(\ref{equn_RF3-2-1cpr1})-(\ref{equn_RF3-2-1cpr42})$, $\phi_0\leq \phi$ on $\partial M\times[0,2T]$, $\phi-2<\phi_0< \phi$ on $\partial M\times[2T,\infty)$, and $\frac{\partial\phi_0}{\partial t}(x,t)\geq 0$ for $t\geq0$. Indeed, this can be done by a gluing construction: Using the density of $C^5$ function in $C^{2,1}(\partial M\times[T,\infty))$, we first choose $\tilde{\phi}\in C^5(\partial M\times[T,\infty))$ sufficiently close to $\phi-1+\frac{1}{100}(1-e^{-t})$ in $C^{2,1}(\partial M\times[T,\infty))$ so that
\begin{align*}
\phi-2<\tilde{\phi}<\phi,
 \end{align*}
 and $\frac{\partial}{\partial t}\tilde{\phi}>0$  on $\partial M\times[T,\infty)$ (see the appendix), and then we let $\eta=\eta(t)\in C^{\infty}([0,\infty))$ such that $\eta(t)=0$ for $t\leq T$, $\eta(t)=1$ for $t\geq 2T$, and $\eta'(t)\geq 0$ for all $t\geq0$, and we define \begin{align*}
 \phi_0(x,t)=\tilde{\phi}(x,t) \eta(t)+ (\beta-3) (1-\eta(t)).
  \end{align*}
  for $(x,t)\in \partial M\times [0,\infty)$. Thus, $\phi_0(x,t)\geq \log(\frac{\xi(t)}{2})$ for $t\geq \max\{2T,T_1\}$. Notice that $\frac{\xi}{2}$ is also a low-speed increasing function satisfying $(\ref{inequn_lowspeedinc})$.

By Theorem \ref{thm_convertibp1-a}, there exists a unique solution $v\in C^{4+\alpha,2+\frac{\alpha}{2}}(\overline{M}\times[0,t])$ for all $t\geq0$ to the following Cauchy-Dirichlet problem
\begin{align*}
&v_t=e^{-2v}\Delta_{h}v-1,\,\text{in}\,\overline{M}\times [0,\infty)\,\\
&v=\phi_0(\cdot,t),\,\,\text{on}\,\partial M \times [0,\infty),\\
&u\big|_{t=0}=v_0,\,\,\text{in}\,M,
\end{align*}
and $v(x,t)$ converges locally uniformly in $C^4$ in $M$ to the Loewner-Nirenberg solution $u_{LN}$ to $(\ref{equn_elliptYb})$ in $M$, as $t\to\infty$. Since $v\leq u$ on $\overline{M}\times\{0\} \bigcup \partial M\times[0,\infty)$, by Lemma \ref{lem_comparisonCDP1}, we have that
\begin{align}
u(x,t)\geq v(x,t)
\end{align}
for $(x,t)\in \overline{M}\times[0,\infty)$.

For the upper bound estimates of $u$, we will use the same argument above Theorem \ref{thm_comparison2}. Indeed, let $u_{LN}$ be the Loewner-Nirenberg solution to $(\ref{equn_elliptYb})$ in $M$. Let $\xi(t)=\displaystyle\sup_{x\in M}(u(x,t)-u_{LN}(x))$ for $t\geq0$. Replacing $U$ by $M$, by the discussion above Theorem \ref{thm_comparison2}, we have that $\displaystyle\limsup_{t\to\infty}\xi(t)\leq0$ and hence,
\begin{align*}
\limsup_{t\to\infty}u(x,t)\leq u_{LN}(x),
\end{align*}
for any $x\in M$. Combining the upper bound estimates, the above lower bound estimates and the convergence of $v$, we have that $u(x,t)\to u_{LN}(x)$ locally uniformly on $M$ as $t\to\infty$. Then by the standard interior estimates of the parabolic equations, we have that $u(x,t)\to u_{LN}(x)$ locally uniformly in $C^2$ sense in $M$ as $t\to\infty$.

\end{proof}
As a corollary, we now prove Theorem \ref{thm_CDPconvergnl}.
\begin{proof}[Proof of Theorem \ref{thm_CDPconvergnl}]
By the discussion at the beginning of Section \ref{section_comparisonthm}, there exists a conformal metric $h=e^{2u_1}g$ such that $K_h=0$ on $\overline{M}$. Let $v=u-u_1$. Then $v$ is a solution to $(\ref{equn_RF3-2-1})-(\ref{equn_ibp3-2-1})$ with the new initial data $v_0=u_0-u_1$ on $\overline{M}$ and boundary data $\tilde{\phi}=\phi-u_1$. Then, by Theorem \ref{thm_Flow1convergence}, we complete the proof of the Theorem.
\end{proof}

\begin{Remark}
The condition $\phi_t\geq 0$ for $t\geq 0$ for the boundary data $\phi$ is too restrictive in Theorem 1.1 in \cite{GLi}, considering the compatible condition (2.1) there. Similar as Theorem \ref{thm_convertibp1-a} here, it can be relaxed to $\phi_t\geq 0$ for $t\geq t_1$ with some constant $t_1>0$, with minor changes in the proof. So now Theorem 1.1 in \cite{GLi} is refined to the following statement:
\end{Remark}
\begin{thm}\label{mainthm2-1}
Let $(M^n, g)$  be a smooth compact manifold with boundary such that $R_g\geq0$. Assume two positive functions $u_0\in C^{2,\alpha}(M)$ and $\phi\in C_{loc}^{2,\alpha}(\partial M \times [0,+\infty))$ satisfy the compatible condition $(2.1)$ on $\partial M\times\{0\}$.  Moreover, assume $\phi_t\geq0$ for $t\geq t_1$ with some constant $t_1>0$, $\phi$ satisfies $(2.4)$ as $t\to\infty$ and $\displaystyle\lim_{t\to\infty}\inf_{\partial M}\phi=\infty$. Then there exists a unique solution $u$ to the Cauchy-Dirichlet problem $(1.3)-(1.5)$, converging in $C_{loc}^2(M^{\circ})$ to a solution $u_{\infty}$ to the Loewner-Nirenberg problem $(1.1)-(1.2)$ as $t\to +\infty$. Moreover, there exists a constant $C>0$ such that
\begin{align*}
\frac{1}{C}(x+(\inf_{p\in \partial M}\phi(p))^{\frac{2}{2-n}})^{\frac{2-n}{2}}-C\leq u\leq Cx^{\frac{2-n}{2}},
\end{align*}
near the boundary $\partial M$, where $x$ is the distance function to the boundary.
\end{thm}
In comparison to the proof of Theorem 1.1 in \cite{GLi}, here for Theorem \ref{mainthm2-1}, we choose the small constant $\epsilon< \frac{1}{10n}\displaystyle\inf_{\partial M\times[0,t_1]}\phi$ instead. That is the only change in the proof.

For the same reason, the condition $\phi_t\geq 0$ for $t\geq0$ in Lemma 2.9 in \cite{GLi} is relaxed to $\phi_t\geq 0$ for $t\geq t_1$ with some constant $t_1>0$, and hence Lemma 2.9 in \cite{GLi} is refined to the following lemma.
\begin{lem}\label{lem2-2}
Assume $(M,g)$ is a compact Riemannian manifold with boundary of $C^{4,\alpha}$ and $R_g=-n(n-1)$. For any $u_0\in C^{2,\alpha}(M)$ such that $u_0>1$ and a positive function $\phi\in C_{loc}^{2+\alpha,1+\frac{\alpha}{2}}(\partial M\times[0,+\infty))$ satisfying the compatible condition $(2.1)$ with $u_0$, $\phi>1$ for $t\geq0$, $\phi_t\geq0$ for $t\geq t_1$ with some constant $t_1>0$ and satisfying $(2.4)$ as $t\to\infty$ and $\displaystyle\lim_{t\to\infty}\inf_{\partial M}\phi=\infty$, there exists a unique positive solution $u\in C_{loc}^{2+\alpha,1+\frac{\alpha}{2}}(M\times[0,+\infty))$ to $(2.13)$, and $u$ converges to $u_{\infty}$ in $C_{loc}^2(M^{\circ})$ as $t\to+\infty$, where $u_{\infty}$ is the solution to the Loewner-Nirenberg problem $(1.1)-(1.2)$.
\end{lem}
Notice that for any positive initial data $u_0\in C^{2,\alpha}(M)$, such a function $\phi$ satisfying the conditions in Lemma \ref{lem2-2} always exists. No changes are needed in the proof of Lemma 2.9 and Theorem 1.2 in \cite{GLi}.

\section{Asymptotic behavior of the geodesic curvature of $\partial M$ under the Cauchy-Dirichlet problem  of the normalized Ricci flow}\label{section_Asympbahavgeocurv1}

Let $(\overline{M},g)$ be a compact surface with its interior $M$ and boundary $\partial M$. As discussed in the proof of Lemma \ref{lem_lowerboundbdary}, there exists $\delta_1\in(0,1)$ small such that the exponential map $F: \partial M\times[0,\delta_1]\to \overline{M}$, which maps $(q,r)$ to the point Exp$_{q}(-r n)$ on the geodesic starting from $q$ along the direction of the inner normal vector $-n$ at $q$, is a diffeomorphism to the image $U_{\delta_1}=\{x\in \overline{M}\big|r(x)\leq \delta_1\}$ with $r(x)$ the distance from $x$ to $\partial M$. Here $q\in\partial M$ realizes the distance $r$ of Exp$_{q}(-r n)$ to $\partial M$. Thus $g$ has the orthogonal decomposition on $U_{\delta_1}$,
\begin{align*}
g=dr^2+g_r=dr^2+f(r,s)ds^2,
\end{align*}
with $f>0$ smooth in $U_{\delta_1}$.

We consider the asymptotic behavior of the geodesic curvature on the boundary along the Cauchy-Dirichlet problem $(\ref{equn_RF3-2})-(\ref{equn_ibp3-2})$ of the normalized Ricci flow, as $t\to+\infty$. Under the assumption of Theorem \ref{thm_CDPconvergnl}, we have proved that there exists a unique solution $u$ for all time $t\geq0$ and $u$ converges to $u_{LN}$ locally uniformly in $C^2$ sense in $M$. We always assume that the condition in Theorem \ref{thm_CDPconvergnl} holds in this section. We will show that when the Dirichlet boundary data $\phi$ grows slowly to infinity, then the geodesic curvature $k_{e^{2u}g}\to1$ uniformly on $\partial M$ as $t\to+\infty$; while when $\phi$ grows sufficiently fast to infinity, the geodesic curvature $k_{e^{2u}g}\to+\infty$ uniformly in a certain speed on $\partial M$ as $t\to+\infty$.

We now assume that the boundary data $\phi$ depends only on $t$ when $t\geq T_3$ for some $T_3>0$ large and $\phi'(t)\to0$ as $t\to+\infty$. We first define a test function $\xi(x,t)$ on $U_{\delta_1}\times[T_3,\infty)$ with $T_3>0$ to be fixed such that
\begin{align}
\xi(x,t)=-\log(a r(x)+\varepsilon(t))\pm w(x),
\end{align}
where $r(x)$ is the distance function to $\partial M$, $\varepsilon=e^{-\phi(t)}$ and $a>0$ is a constant to be determined, and
\begin{align}
w(x)=A((r(x)+\delta)^{-p}-\delta^{-p})
\end{align}
with two large constants $A>1$ and $p>1$ and a small constant $\delta>0$ to be determined. Hence, $w(x)\leq 0$ on $U_{\delta_1}$ with $w(x)=0$ on $\partial M$. Therefore, we have the following computation:
\begin{align*}
\frac{\partial \xi}{\partial t}&=-\frac{\varepsilon'(t)}{a r(x)+\varepsilon(t)}=\frac{\varepsilon(t)\phi'(t)}{a r(x)+\varepsilon(t)},\\
\Delta_g\xi(x,t)&=\frac{\partial^2\xi}{\partial r^2}+\frac{1}{2}f^{-1}\frac{\partial f}{\partial r}\frac{\partial \xi}{\partial r}\\
&=\frac{a^2}{(a r(x)+\varepsilon(t))^2}\pm Ap(p+1)(r+\delta)^{-p-2}+\frac{1}{2}f^{-1}\frac{\partial f}{\partial r}[-\frac{a}{a r(x)+\varepsilon(t)}\mp Ap(r(x)+\delta)^{-p-1}].
\end{align*}
Therefore,
\begin{align*}
&e^{-2\xi}(\Delta_g\xi-K_g)-1\\
&=e^{\mp 2w}[a^2\pm\frac{Ap(p+1)(a r(x)+\varepsilon(t))^2}{(r+\delta)^{p+2}}+\frac{1}{2}f^{-1}\frac{\partial f}{\partial r}\big(-a(a r(x)+\varepsilon(t))\mp\frac{Ap(a r(x)+\varepsilon(t))^2}{(r(x)+\delta)^{p+1}}\big)\\
&\,\,\,\,\,\,\,\,\,\,\,\,\,\,\,\,\,-K_g(a r(x)+\varepsilon(t))^2]-1,
\end{align*}
For any $a\in(1,2)$ close to $1$, we define $\underline{u}=-\log(a r(x)+\varepsilon(t))+ w(x)$ on $U_{\delta_1}\times[T_3,\infty)$. Now fix $\delta>0$ small. We then choose $p>1$ large so that
\begin{align*}
p+1> 2(\delta_1+\delta) \displaystyle\sup_{U_{\delta_1}}\big|f^{-1}\frac{\partial f}{\partial r}\big|,
\end{align*}
and
\begin{align*}
p(p+1)\geq 2(\delta_1+\delta)^2\sup_{\overline{M}}|K_g|,
\end{align*}
and hence,
\begin{align}
e^{-2\underline{u}}(\Delta_g\underline{u}-K_g)-1\geq a^2-1-a(ar(x)+\varepsilon(t))\,\sup_{U_{\delta_1}}\big|f^{-1}\frac{\partial f}{\partial r}\big|,
\end{align}
on $U_{\delta_1}\times[T_3,\infty)$. Since $\phi(t)\to\infty $ as $t\to\infty$, there exists $T_4>T_3$ such that
\begin{align*}
\varepsilon(t)=e^{-\phi(t)}<(a^2-1)[4a\displaystyle\sup_{U_{\delta_1}}\big|f^{-1}\frac{\partial f}{\partial r}\big|]^{-1},
\end{align*}
and $\phi(t)>0$ for $t\geq T_4$. Let $\delta_2\in(0,\delta_1)$ satisfy that $\delta_2<\frac{1}{4} (1-a^{-2})$. Denote $U_{\delta_2}=\{x\in U_{\delta_1}\big|r(x)\leq \delta_2\}$. Then, we take $T_4$ large so that $\phi'(t)\leq \frac{a^2-1}{4}$ for $t\in[T_4,\infty)$, and hence we have that
\begin{align*}
\underline{u}_t\leq e^{-2\underline{u}}(\Delta_g\underline{u}-K_g)-1,
\end{align*}
on $U_{\delta_2}\times [T_4,\infty)$. Moreover, by the locally uniform convergence of $u$, we also require $T_4$ to be large so that
\begin{align*}
\frac{11}{10}u_{NL}(x)\geq u(x,t)\geq \frac{9}{10}u_{NL}(x),
\end{align*}
 and hence $u$ is uniformly bounded from the above and below, on $\Sigma_{\delta_2}\times[T_4,\infty)$ with $\Sigma_{\delta_2}=\{x\in M\big|\,r(x)=\delta_2\}$. Notice that
 \begin{align*}
 -\log(a\delta_2+1)\leq-\log(a\delta_2+e^{-\phi(t)})\leq -\log(a\delta_2).
 \end{align*}
  Now we choose $A>1$ and $p>1$ large so that
\begin{align*}
\underline{u}\leq u,
\end{align*}
on $\Sigma_{\delta_2}\times[T_4,\infty)$. Since $u(\cdot, T_4)$ and $\underline{u}(\cdot,T_4)$ are continuous on $U_{\delta_2}$, there exists $A_1>1$ and $p_1>1$ such that for $A>A_1$ and $p>p_1$, we have that
\begin{align*}
\underline{u}\leq u
\end{align*}
on $U_{\delta_2}\times\{T_4\}$. Then, we apply Lemma \ref{lem_comparisonCDP1} on $U_{\delta_2}\times [T_4,\infty)$ to have that
\begin{align*}
\underline{u}\leq u
\end{align*}
on $U_{\delta_2}\times [T_4,\infty)$. Recall that $\underline{u}=u$ on $\partial M\times[T_4,\infty)$, and hence,
\begin{align*}
k_{e^{2u}g}=e^{-\phi(t)}[\frac{\partial u}{\partial n_g}+k_{g}]&\leq e^{-\phi(t)}[\frac{\partial \underline{u}}{\partial n_g}+k_g]\\
&=e^{-\phi(t)}[-\frac{\partial \underline{u}}{\partial r}\big|_{r=0}+k_g]=a+[k_g+Ap\delta^{-p-1}]e^{-\phi(t)},
\end{align*}
on $\partial M\times[T_4,\infty)$, where $n_g$ is the outer normal vector fields on $\partial M$. Similarly, for any $b\in(0,1)$ close to $1$, we define \begin{align}\label{equntf}
\overline{u}(x,t)=-\log(b\,r(x)+\varepsilon(t))-w(x),
\end{align}
and then there exists $T_4'>T_3$, $\delta_2'\in(0,\delta_1)$, and $A'>1$ and $p'>1$ large which depend on $b$ so that
\begin{align*}
\overline{u}\geq u,
\end{align*}
on $U_{\delta_2'}\times [T_4',\infty)$; moreover, we have
\begin{align*}
k_{e^{2u}g}=e^{-\phi(t)}[\frac{\partial u}{\partial n_g}+k_{g}]\geq e^{-\phi(t)}[\frac{\partial \overline{u}}{\partial n_g}+k_g]= b+[k_g-Ap\delta^{-p-1}]e^{-\phi(t)},
\end{align*}
on $\partial M\times[T_4',\infty)$. Now let $a\to 1$ and $b\to 1$, and hence we have that if $\phi'(t)\to 0$ as $t\to+\infty$, then
\begin{align*}
k_{e^{2u}g}=e^{-\phi(t)}[\frac{\partial u}{\partial n_g}+k_{g}]\to 1,
\end{align*}
uniformly on $\partial M$ as $t\to\infty$. This completes the proof of the following theorem:
\begin{thm}\label{thm_asymptbdlowspeed1}
Let $(\overline{M},g)$ be a compact surface with its interior $M$ and boundary $\partial M$. Under the condition of Theorem \ref{thm_CDPconvergnl}, we assume that there exists $T_3>0$ such that $\phi$ depends only on $t$, i.e., $\phi=\phi(t)$ for $t\geq T_3$, and $\phi'(t)\to0$ as $t\to\infty$. Then the solution $u(x,t)$ obtained in Theorem \ref{thm_CDPconvergnl}, which converges locally uniformly in $C^2$ to $u_{LN}$, satisfies that the geodesic curvature on $\partial M$
\begin{align*}
k_{e^{2u}g}=e^{-u}(\frac{\partial u}{\partial n_g}+k_g)\to 1,
\end{align*}
uniformly on $\partial M$ as $t\to\infty$.
\end{thm}

One can use the same test functions to estimate the upper and lower bound of $\frac{\partial u}{\partial n_g}$ on $\partial M$ when the boundary data $\phi$ is a certain small perturbation of the function $\phi(t)$ in spatial directions for $t\geq T_3$. Indeed, we have the following generalization:
\begin{thm}\label{thm_lowspeedCDdataperturb}
Let $(\overline{M},g)$ be a compact surface with its interior $M$ and boundary $\partial M$. Under the condition of Theorem \ref{thm_CDPconvergnl}, we assume that there exists $T_3>0$ and a constant $C_1>0$ such that
\begin{align}
e^{-\phi}(|\nabla^2\phi|+|\nabla\phi|^2+|\nabla\phi|)\leq C_1,
\end{align}
for $(x,t)\in \partial M\times[T_3,\infty)$, and $\sup_{\partial M}\frac{\partial\phi}{\partial t}\to0$ as $t\to\infty$. Then the solution $u(x,t)$ obtained in Theorem \ref{thm_CDPconvergnl}, which converges locally uniformly in $C^2$ to $u_{LN}$, satisfies that the geodesic curvature on $\partial M$
\begin{align*}
k_{e^{2u}g}=e^{-u}(\frac{\partial u}{\partial n_g}+k_g)\to 1,
\end{align*}
uniformly on $\partial M$ as $t\to\infty$.
\end{thm}
\begin{proof}
The proof is almost the same as that of Theorem \ref{thm_asymptbdlowspeed1}. We only show the differences.

We extend $\phi(x,t)$ on $\partial M\times[0,\infty)$ to $U_{\delta_1}\times[0,\infty)$ using the exponential map from the boundary in the following way: Recall that the map $F:\, \partial M\times [0,\delta_1]\to \overline{M}$ defined at the beginning of the section such that $F(q,r)=$Exp$_q(-r n)$ is a diffeomorphism to $U_{\delta_1}$, and for each $x\in U_{\delta_1}$, there exists a unique $q=q(x)\in \partial M$ that realizes the distance $r(x)$ from $x$ to $\partial M$. Indeed, $x=F(q(x),r(x))$. Then let $\phi(x,t)=\phi(q(x),t)$ for $(x,t)\in U_{\delta_1}\times[0,\infty)$. In particular, $\frac{\partial \phi}{\partial r}=0$ on $U_{\delta_1}\times[T_3,\infty)$.

We define a test function $\xi(x,t)$ on $U_{\delta_1}\times[T_3,\infty)$ with $T_3>0$ to be fixed such that
\begin{align}
\xi(x,t)=-\log(a r(x)+\varepsilon(x,t))\pm w(x),
\end{align}
where $r(x)$ is the distance function to $\partial M$, $\varepsilon=e^{-\phi(x,t)}$ and $a>0$ is a constant to be determined, and
\begin{align}
w(x)=A((r(x)+\delta)^{-p}-\delta^{-p}),
\end{align}
with two large constants $A>1$ and $p>1$ and a small constant $\delta>0$ to be determined.

 Then we have that
\begin{align*}
\frac{\partial \xi}{\partial t}&=\frac{\varepsilon(x,t)\phi_t(x,t)}{a r(x)+\varepsilon(x,t)},\\
\Delta_g\xi(x,t)&=\frac{\partial^2\xi}{\partial r^2}+\frac{1}{2}f^{-1}\frac{\partial f}{\partial r}\frac{\partial \xi}{\partial r}+f^{-1}\xi_{ss}-\frac{1}{2}f^{-1}f_s\xi_s\\
&=\frac{a^2}{(a r(x)+\varepsilon(t))^2}\pm Ap(p+1)(r+\delta)^{-p-2}+\frac{1}{2}f^{-1}\frac{\partial f}{\partial r}[-\frac{a}{a r(x)+\varepsilon(t)}\mp Ap(r(x)+\delta)^{-p-1}]\\
&\,\,\,\,\,\,+\frac{\varepsilon^2(\phi_s)^2}{(ar+\varepsilon(x,t))^2f}+\frac{\varepsilon\,(\phi_{ss}-(\phi_s)^2)}{(ar+\varepsilon(x,t))f}-\frac{f_s\,\varepsilon\,\phi_s}{2\,(ar+\varepsilon(x,t))f}.
\end{align*}
Therefore,
\begin{align*}
&e^{-2\xi}(\Delta_g\xi-K_g)-1\\
&=e^{\mp 2w}[a^2\pm\frac{Ap(p+1)(a r(x)+\varepsilon(t))^2}{(r+\delta)^{p+2}}+\frac{1}{2}f^{-1}\frac{\partial f}{\partial r}\big(-a(a r(x)+\varepsilon(t))\mp\frac{Ap(a r(x)+\varepsilon(t))^2}{(r(x)+\delta)^{p+1}}\big)\\
&\,\,\,\,\,\,\,\,\,\,\,\,\,\,\,\,\,-K_g(a r(x)+\varepsilon(t))^2+f^{-1}\varepsilon^2(\phi_s)^2+f^{-1}\varepsilon\,(\phi_{ss}-(\phi_s)^2)(ar+\varepsilon(x,t))\\
&\,\,\,\,\,\,\,\,\,\,\,\,\,\,\,\,\,-(2f)^{-1}\,f_s\,\varepsilon\,\phi_s\,(ar+\varepsilon(x,t))]-1,
\end{align*}
Let $C=\sup_{U_{\delta_1}}(f^{-1}+(2f)^{-1}|f_s|)$. For any $a\in(1,2)$ close to $1$, we define
\begin{align*}
\underline{u}=-\log(a r(x)+\varepsilon(x,t))+ w(x),
\end{align*}
on $U_{\delta_1}\times[T_3,\infty)$. Now fix $\delta>0$ small. We then choose $p>1$ large so that
\begin{align*}
p+1> 2(\delta_1+\delta) \displaystyle\sup_{U_{\delta_1}}\big|f^{-1}\frac{\partial f}{\partial r}\big|,
\end{align*}
and
\begin{align*}
p(p+1)\geq 2(\delta_1+\delta)^2\sup_{\overline{M}}|K_g|,
\end{align*}
and hence, using the fact $w(x)\leq0$ on $U_{\delta_1}$, we have
\begin{align}
e^{-2\underline{u}}(\Delta_g\underline{u}-K_g)-1\geq a^2-1-a(ar(x)+\varepsilon(x,t))\,\sup_{U_{\delta_1}}\big|f^{-1}\frac{\partial f}{\partial r}\big|-C(e^{-\phi}+ar)e^{-\phi}(|\phi_{ss}|+|\phi_s|^2+|\phi_s|).
\end{align}
 Recall that $\inf_{\partial M}\phi(\cdot,t)\to +\infty$ as $t\to\infty$, and
 \begin{align*}
 e^{-\phi}(|\phi_{ss}|+|\phi_s|^2+|\phi_s|)\leq C_1(1+\sup_{U_{\delta_2}}(f+f^{-1}|f_s|)),
 \end{align*}
 for some constant $C_1$ for $t\geq T_3$. For this $a\in(1,2)$, we choose $\delta_2\in(0,\delta_1)$ so that
 \begin{align*}
 \delta_2<\min\{\frac{a^2-1}{8CC_1a(1+\sup_{U_{\delta_2}}(f+f^{-1}|f_s|))},\,\frac{1}{8} (1-a^{-2})(1+\sup_{U_{\delta_1}}\big|f^{-1}\frac{\partial f}{\partial r}\big|)^{-1}\}.
 \end{align*}
 Also, since $\inf_{\partial M}\phi \to +\infty$ as $t\to+\infty$, let $T_4>T_3$ be large so that
\begin{align*}
CC_1(1+\sup_{U_{\delta_2}}(f+f^{-1}|f_s|))e^{-\phi}<\frac{1}{8}(a^2-1),
\end{align*}
and $\phi(x,t)>0$, and moreover,
\begin{align*}
\varepsilon(x,t)=e^{-\phi(x,t)}<(a^2-1)[8a\displaystyle\sup_{U_{\delta_1}}\big|f^{-1}\frac{\partial f}{\partial r}\big|]^{-1},
\end{align*}
for $t\geq T_4$. Denote $U_{\delta_2}=\{x\in U_{\delta_1}\big|r(x)\leq \delta_2\}$. Then, we also take $T_4$ large so that $\phi_t(x,t)\leq \frac{a^2-1}{8}$ for $t\in[T_4,\infty)$, and hence we have that
\begin{align*}
\underline{u}_t\leq e^{-2\underline{u}}(\Delta_g\underline{u}-K_g)-1,
\end{align*}
on $U_{\delta_2}\times [T_4,\infty)$. Then using the same argument as in the proof of Theorem \ref{thm_asymptbdlowspeed1}, we can choose $A>1$ and $p>1$ large, depending on $a>1$, such that $\underline{u}\leq u$ on $\Sigma_{\delta_2}\times[T_4,\infty) \bigcup U_{\delta_2}\times\{T_4\}$ and $\underline{u}=u$ on $\partial M\times[T_4,\infty)$ for $T_4$ large. Then, we apply Lemma \ref{lem_comparisonCDP1} on $U_{\delta_2}\times [T_4,\infty)$ to have
\begin{align*}
\underline{u}\leq u,
\end{align*}
on $U_{\delta_2}\times [T_4,\infty)$, and hence,
\begin{align*}
k_{e^{2u}g}=e^{-\phi}[\frac{\partial u}{\partial n_g}+k_{g}]&\leq e^{-\phi}[\frac{\partial \underline{u}}{\partial n_g}+k_g]\\
&=e^{-\phi}[-\frac{\partial \underline{u}}{\partial r}\big|_{r=0}+k_g]=a+[k_g+Ap\delta^{-p-1}]e^{-\phi},
\end{align*}
on $\partial M\times[T_4,\infty)$, where $n_g$ is the outer normal vector fields on $\partial M$ in $(\overline{M},g)$. Similarly, for any $b\in(0,1)$ close to $1$, define $\overline{u}(x,t)=-\log(b\,r(x)+\varepsilon(x,t))-w(x)$, and then there exists $T_4'>T_3$, $\delta_2'\in(0,\delta_1)$, and $A'>1$ and $p'>1$ large which depend on $b$ so that
\begin{align*}
\overline{u}\geq u,
\end{align*}
on $U_{\delta_2'}\times [T_4',\infty)$; moreover, we have
\begin{align*}
k_{e^{2u}g}=e^{-\phi}[\frac{\partial u}{\partial n_g}+k_{g}]\geq e^{-\phi}[\frac{\partial \overline{u}}{\partial n_g}+k_g]= b+[k_g-Ap\delta^{-p-1}]e^{-\phi},
\end{align*}
on $\partial M\times[T_4',\infty)$. Now let $a\to 1$ and $b\to 1$, and hence we have that
\begin{align*}
k_{e^{2u}g}=e^{-\phi}[\frac{\partial u}{\partial n_g}+k_{g}]\to 1,
\end{align*}
uniformly on $\partial M$ as $t\to\infty$.

\end{proof}

We now show that if $\phi$ depends only on $t$ for $t>0$ large and $\phi'(t)\to +\infty$ as $t\to+\infty$, then the solution $u$ to Theorem \ref{thm_CDPconvergnl} satisfies that $k_{e^{2u}g}\to +\infty$ as $t\to\infty$.
\begin{thm}\label{thm_asymptbdfastspeed1}
Let $(\overline{M},g)$ be a compact surface with its interior $M$ and boundary $\partial M$. Under the condition of Theorem \ref{thm_CDPconvergnl}, we assume that there exists $T_3>0$ such that $\phi$ depends only on $t$, i.e., $\phi=\phi(t)$ for $t\geq T_3$, and $\phi'(t)\to+\infty$ as $t\to\infty$. Then the solution $u(x,t)$ obtained in Theorem \ref{thm_CDPconvergnl}, which converges locally uniformly in $C^2$ to $u_{LN}$ on $M$, satisfies that the geodesic curvature on $\partial M$
\begin{align*}
k_{e^{2u}g}=e^{-u}(\frac{\partial u}{\partial n_g}+k_g)\to +\infty,
\end{align*}
uniformly on $\partial M$ as $t\to\infty$.
\end{thm}
\begin{proof}
Notice that for the Loewner-Nirenberg problem $(\ref{equn_LNFn2})-(\ref{equn_LNBDn2})$, paralleling to the approach of Loewner and Nirenberg for higher dimension cases, one solves the sequence of Dirichlet boundary value problems for $N\in \mathbb{N}$:
\begin{align*}
&-\Delta_gu_N+K_g=-e^{2u_N},\,\,\,\text{in}\,\,M,\\
&u_N= N,\,\,\,\text{on}\,\,\partial M.
\end{align*}
Let $u_N$ be the unique solution to this Dirichlet problem, and using maximum principle we have that $u_N\to u_{LN}$ with $u_{LN}$ locally uniformly in $C^2$ sense in $M$. Using the argument in Theorem \ref{thm_asymptbdlowspeed1}, with $\varepsilon$ replaced by $e^{-N}$ in the test functions $\underline{u}$ and $\overline{u}$, we have the geodesic curvature $k_{e^{2u_N}g}\to 1$ as $N\to+\infty$. Therefore, since the problem $(\ref{equn_RF3-2})-(\ref{equn_ibp3-2})$ is conformally invariant, we choose the background metric $h=e^{2u_N}g$ in the flow for some $N>0$ large so that the geodesic curvature
\begin{align*}
k_{h}\in(\frac{9}{10},\frac{10}{9}).
\end{align*}
Now let $v=u-u_N$ on $\overline{M}\times[0,\infty)$, $\tilde{\phi}(x,t)=\phi(x,t)-u_N(x)=\phi(x,t)-N$ on $\partial M\times[0,\infty)$, and $v_0(x)=u_0(x)-u_N(x)$ on $\overline{M}$. Hence, we have
\begin{align}\label{equn_CDP2-2}
&v_t=e^{-2v}(\Delta_{h}v-K_{h})-1,\,\text{in}\,M\times [0,\infty)\,\\
&v=\tilde{\phi}(\cdot,t),\,\,\text{on}\,\partial M \times [0,\infty),\\
&\label{equn_CDP2-2-CD}v\big|_{t=0}=v_0,\,\,\text{in}\,M.
\end{align}
Now for $(\overline{M},h)$, in the neighborhood $U_{\delta_1}$ of the boundary for $\delta_1>0$ small so that $\partial M\times[0,\delta_1]$ is diffeomorphic to $U_{\delta_1}$ under the exponential map from the boundary along the direction of the inner normal vector fields, the metric has the orthogonal decomposition
\begin{align*}
h=dr^2+h_r=dr^2+f(r,s)ds^2,
\end{align*}
where $r(x)$ is the distance of $x$ to $\partial M$ in $(\overline{M},h)$, with $h_r$ the restriction of $h$ on $\Sigma_r=\{p\in M\big| r(p)=r\}$.  We extend $\tilde{\phi}$ to $U_{\delta_1}\times[T_3,\infty)$, so that $\tilde{\phi}(x,t)=\phi(t)-N$ for $(x,t)\in U_{\delta_1}\times[T_3,\infty)$. Now, for any constant $b>1$, we modify the test function $(\ref{equntf})$, and define
\begin{align*}
\overline{v}(x,t)=-\log(b r(x)+\varepsilon)-w(x)+\,e^{\phi(t)-N}r(x),
\end{align*}
where $\varepsilon=\varepsilon(t)=e^{-\tilde{\phi}}=e^{-\phi(t)+N}$ and $w(x)=A[(r+\delta)^{-p}-\delta^{-p}]\leq 0$ in $U_{\delta_1}$, with any $\delta>0$ small fixed, $A>1$ and $p>1$ to be determined. Then we have the computation
\begin{align*}
\frac{\partial \overline{v}}{\partial t}&=-\frac{\varepsilon'(t)}{b r(x)+\varepsilon(t)}+\,e^{\phi(t)-N}r(x)\phi'=\frac{\varepsilon\phi'(t)}{b r(x)+\varepsilon(t)}+\,e^{\phi(t)-N}r(x)\phi',\\
\Delta_h\overline{v}(x,t)&=\frac{\partial^2\overline{v}}{\partial r^2}+\frac{1}{2}f^{-1}\frac{\partial f}{\partial r}\frac{\partial \overline{v}}{\partial r}\\
&=\frac{b^2}{(b r(x)+\varepsilon(t))^2}- \frac{Ap(p+1)}{(r+\delta)^{p+2}}+\frac{1}{2}f^{-1}\frac{\partial f}{\partial r}[\frac{-b}{b r(x)+\varepsilon(t)}+ Ap(r(x)+\delta)^{-p-1}+\,e^{\phi(t)-N}],
\end{align*}
and hence,
\begin{align*}
e^{-2\overline{v}}[\Delta_h\overline{v}-K_h]-1&=e^{2w-2\,r\,e^{\phi-N}}[\frac{1}{2}f^{-1}\frac{\partial f}{\partial r}\big(-b(b r+\varepsilon)+ \frac{Ap(b r+\varepsilon)^2}{(r(x)+\delta)^{p+1}}+\,e^{\phi(t)-N}(b r+\varepsilon)^2\big)\\
&\,\,\,\,\,\,\,\,\,\,\,\,\,\,\,\,\,\,\,\,\,\,\,\,\,\,\,\,\,+b^2- \frac{Ap(p+1)(b r+\varepsilon)^2}{(r+\delta)^{p+2}}-K_h(b r+\varepsilon)^2]-1.
\end{align*}
 We require that
\begin{align*}
p+1> (\delta_1+\delta) \displaystyle\sup_{U_{\delta_1}}\big|f^{-1}\frac{\partial f}{\partial r}\big|+ (\delta_1+\delta)^2\sup_{\overline{M}}|K_h|,
\end{align*}
and hence,
\begin{align*}
e^{-2\overline{v}}[\Delta_h\overline{v}-K_h]-1&\leq e^{2w-2\,r\,e^{\phi-N}}[\frac{1}{2}f^{-1}\frac{\partial f}{\partial r}\big(-b(b r+\varepsilon)+\,e^{\phi(t)-N}(b r+\varepsilon)^2\big)+b^2]-1.
\end{align*}
Recall that
\begin{align*}
k_h=-\frac{1}{2}f^{-1}\frac{\partial f}{\partial r}\in(\frac{9}{10},\frac{10}{9})
 \end{align*}
 on $\partial M$, and by continuity, we choose $\delta_2\in(0,\delta_1)$ small so that
\begin{align*}
\frac{1}{2}f^{-1}\frac{\partial f}{\partial r}\in(-2,-\frac{1}{2})
\end{align*}
in $U_{\delta_2}=\{x\in U_{\delta_1}\big|r(x)\leq \delta_2\}$, and hence,
\begin{align*}
e^{-2\overline{v}}[\Delta_h\overline{v}-K_h]-1&\leq e^{2w-2\,r\,e^{\phi-N}}[-\frac{b}{2}(b r+\varepsilon)f^{-1}\frac{\partial f}{\partial r}+b^2]-1,
\end{align*}
for $(x,t)\in U_{\delta_2}\times[T_3,\infty)$. Since $\phi'(t)\to+\infty$ as $t\to\infty$, we choose $T_4\geq T_3$  large so that $\phi'(t)>0$ and
\begin{align*}
\varepsilon(t)=e^{-\phi(t)+N}\leq \frac{b}{4},
\end{align*}
for $(x,t)\in U_{\delta_2}\times[T_4,\infty)$. Also, let $\delta_2<\frac{1}{4}$. Therefore,
\begin{align*}
e^{-2\overline{v}}[\Delta_h\overline{v}-K_h]-1&\leq 2b^2 e^{2w-2\,r\,e^{\phi-N}}-1\leq 2b^2-1,
\end{align*}
for $(x,t)\in U_{\delta_2}\times[T_4,+\infty)$. On the other hand, for $0<r<b^{-1}e^{-\phi(t)+N}$, we have
\begin{align*}
\frac{\partial \overline{v}}{\partial t}\geq\frac{\varepsilon(t)\phi'(t)}{b r(x)+\varepsilon(t)}\geq \frac{1}{2}\phi'(t);
\end{align*}
 while for $b^{-1}e^{-\phi(t)+N} < r\leq \delta_2$,
\begin{align*}
\frac{\partial \overline{v}}{\partial t}\geq \,e^{\phi(t)-N}r\phi'(t)\geq b^{-1}\phi'(t).
\end{align*}
Recall that $\phi'(t)\to+\infty$ as $t\to\infty$. Hence, we choose $T_5>T_4$ to be large so that
\begin{align*}
\frac{\partial \overline{v}}{\partial t}\geq e^{-2\overline{v}}[\Delta_h\overline{v}-K_h]-1,
\end{align*}
for $(x,t)\in U_{\delta_2}\times[T_5,\infty)$. Moreover, we require $T_5$ be large so that $v$ is sufficiently close to the Loewner-Nirenberg solution $v_{LN}=u_{LN}-u_N$ on $(\overline{M},h)$, on $\Sigma_{\delta_2}=\{x\in U_{\delta_2}\big|r(x)=\delta_2\}$ for $t\geq T_5$. Then as in the proof of Theorem \ref{thm_asymptbdlowspeed1}, we choose $A>1$ and $p>1$ large so that $\overline{v}\geq v$ on $\Sigma_{\delta_2}\times[T_5,\infty) \bigcup U_{\delta_2}\times\{T_5\}$. Then, we apply Lemma \ref{lem_comparisonCDP1} on $U_{\delta_2}\times [T_5,\infty)$ to have that
\begin{align*}
\overline{v}\geq v,
\end{align*}
on $U_{\delta_2}\times [T_5,\infty)$. Since $\overline{v}= v$ on $\partial M\times [T_5,\infty)$, we have
\begin{align*}
k_{e^{2u}g}=k_{e^{2v}h}=e^{-\tilde{\phi}}[\frac{\partial v}{\partial n_h}+k_{h}]&\geq e^{-\tilde{\phi}}[\frac{\partial \overline{v}}{\partial n_h}+k_h]\\
&=e^{-\phi(t)+N}[-\frac{\partial \overline{v}}{\partial r}\big|_{r=0}+k_h]=b-1+[k_h-Ap\delta^{-p-1}]e^{-\phi(t)+N},
\end{align*}
on $\partial M\times[T_5,\infty)$, where $n_h$ is the outer normal vector fields on $\partial M$ on $(\overline{M},h)$. Therefore, let $b\to+\infty$, we have that
\begin{align*}
k_{e^{2u}g}\to +\infty,
\end{align*}
uniformly on $\partial M$ as $t\to\infty$.

\end{proof}
Now we are ready to estimate the geodesic curvature on the boundary when the Dirichlet boundary data grows even faster.
\begin{thm}\label{thm_asymptbdfastspeed2}
Let $(\overline{M},g)$ be a compact surface with its interior $M$ and boundary $\partial M$. Under the condition of Theorem \ref{thm_CDPconvergnl}, we assume that there exists $T_3>0$ such that $\phi$ depends only on $t$, i.e., $\phi=\phi(t)$, and
\begin{align}\label{inequn_fastgrowthupperbound}
\phi'(t)\geq 3\phi(t)+1,
\end{align}
for $t\geq T_3$. Then the solution $u(x,t)$ obtained in Theorem \ref{thm_CDPconvergnl}, which converges locally uniformly in $C^2$ to $u_{LN}$ on $M$, satisfies that the geodesic curvature on $\partial M$ has the lower bound estimate
\begin{align*}
k_{e^{2u}g}=e^{-u}(\frac{\partial u}{\partial n_g}+k_g)\geq \phi^{\frac{1}{3}}-1-C\,e^{-\phi(t)},
\end{align*}
on $\partial M\times[T_4,\infty)$ for some constant $C>0$ and some $T_4\geq T_3$, where $n_g$ is the outer normal vector fields on $\partial M$ on $(\overline{M},g)$.
\end{thm}
\begin{proof}
As in the proof of Theorem \ref{thm_asymptbdfastspeed1}, since the form of the flow is conformally invariant, we assume the background metric $h=e^{2u_N}g$ to have the geodesic curvature $k_h\in(\frac{9}{10},\frac{10}{9})$. The same as the proof of Theorem \ref{thm_asymptbdfastspeed1}, we let $v=u-u_N$ on $\overline{M}\times[0,\infty)$, $\tilde{\phi}(x,t)=\phi(x,t)-N$ on $\partial M\times[0,\infty)$, and $v_0(x)=u_0(x)-u_N(x)$ on $\overline{M}$ and hence, $v$ satisfies $(\ref{equn_CDP2-2})-(\ref{equn_CDP2-2-CD})$; moreover, in the neighborhood $U_{\delta_1}$ of $\partial M\subseteq (\overline{M},h)$ for $\delta_1>0$ small so that $\partial M\times[0,\delta_1]$ is diffeomorphic to $U_{\delta_1}$ under the exponential map from the boundary along the direction of the inner normal vector fields, the metric has the orthogonal decomposition
\begin{align*}
h=dr^2+h_r=dr^2+f(r,s)ds^2,
\end{align*}
where $r(x)$ is the distance of $x$ to $\partial M$ in $(\overline{M},h)$, with $h_r$ the restriction of $h$ on $\Sigma_r=\{p\in M\big| r(p)=r\}$. We extend $\tilde{\phi}$ to $U_{\delta_1}\times[T_3,\infty)$, so that $\tilde{\phi}(x,t)=\phi(t)-N$ for $(x,t)\in U_{\delta_1}\times[T_3,\infty)$. Then we choose the test function $\overline{v}$ on $U_{\delta_1}\times[T_3,\infty)$ such that
\begin{align*}
\overline{v}(x,t)=-\log(b(t) r(x)+\varepsilon(t))-w(x)+\,e^{\phi(t)-N}r(x),
\end{align*}
where $r(x)$ is the distance of $x$ to $\partial M$ in $(\overline{M},g)$,  $b(t)=\phi^{\frac{1}{3}}(t)$, $\varepsilon(t)=e^{-\phi+N}$ and $w(x)=A[(r+\delta)^{-p}-\delta^{-p}]\leq 0$ in $U_{\delta_1}$, with any $\delta>0$ small fixed, $A>1$ and $p>1$ to be determined.  Then we have the computation
\begin{align*}
\frac{\partial \overline{v}}{\partial t}&=-\frac{b'(t)r(x)+\varepsilon'(t)}{b r(x)+\varepsilon(t)}+\,e^{\phi(t)-N}r(x)\phi'=\frac{-\frac{1}{3}r(x)\phi^{\frac{-2}{3}}\phi'+\varepsilon(t)\phi'(t)}{b r(x)+\varepsilon(t)}+\,e^{\phi(t)-N}r(x)\phi',\\
\Delta_h\overline{v}(x,t)&=\frac{\partial^2\overline{v}}{\partial r^2}+\frac{1}{2}f^{-1}\frac{\partial f}{\partial r}\frac{\partial \overline{v}}{\partial r}\\
&=\frac{b^2}{(b r(x)+\varepsilon(t))^2}- \frac{Ap(p+1)}{(r+\delta)^{p+2}}+\frac{1}{2}f^{-1}\frac{\partial f}{\partial r}[\frac{-b}{b r(x)+\varepsilon(t)}+ Ap(r(x)+\delta)^{-p-1}+\,e^{\phi(t)-N}],
\end{align*}
and hence,
\begin{align*}
e^{-2\overline{v}}[\Delta_h\overline{v}-K_h]-1&=e^{2w-2\,r\,e^{\phi-N}}[\frac{1}{2}f^{-1}\frac{\partial f}{\partial r}\big(-b(b r+\varepsilon)+ \frac{Ap(b r+\varepsilon)^2}{(r(x)+\delta)^{p+1}}+\,e^{\phi(t)-N}(b r+\varepsilon)^2\big)\\
&\,\,\,\,\,\,\,\,\,\,\,\,\,\,\,\,\,\,\,\,\,\,\,\,\,\,\,\,\,\,\,\,+b^2- \frac{Ap(p+1)(b r+\varepsilon)^2}{(r+\delta)^{p+2}}-K_h(b r+\varepsilon)^2]-1.
\end{align*}
We require that
\begin{align*}
p+1> (\delta_1+\delta) \displaystyle\sup_{U_{\delta_1}}\big|f^{-1}\frac{\partial f}{\partial r}\big|+ (\delta_1+\delta)^2\sup_{\overline{M}}|K_h|,
\end{align*}
and hence,
\begin{align*}
e^{-2\overline{v}}[\Delta_h\overline{v}-K_h]-1&\leq e^{2w-2\,r\,e^{\phi-N}}[\frac{1}{2}f^{-1}\frac{\partial f}{\partial r}\big(-b(b r+\varepsilon)+\,e^{\phi(t)-N}(b r+\varepsilon)^2\big)+b^2]-1.
\end{align*}
Then the same as in the proof of Theorem \ref{thm_asymptbdfastspeed1}, since $\phi(t)\to+\infty$ as $t\to\infty$, and by the choice that
 \begin{align*}
k_h=-\frac{1}{2}f^{-1}\frac{\partial f}{\partial r}\in(\frac{9}{10},\frac{10}{9})
 \end{align*}
 on $\partial M$, we can choose $\delta_2\in(0,\delta_1)$ small with $\delta_2<\frac{1}{4}$ and $T_4>T_3$ so that $\phi'(t)>0$ for $t\geq T_4$, and also,
\begin{align*}
\frac{1}{2}f^{-1}\frac{\partial f}{\partial r}\in(-2,-\frac{1}{2})
\end{align*}
in $U_{\delta_2}=\{x\in U_{\delta_1}\big|r(x)\leq \delta_2\}$, and hence,
\begin{align*}
e^{-2\overline{v}}[\Delta_h\overline{v}-K_h]-1&\leq 2b^2 e^{2w-2\,r\,e^{\phi-N}}-1\leq 2b^2-1=2\phi^{\frac{2}{3}}-1,
\end{align*}
for $(x,t)\in U_{\delta_2}\times[T_4,+\infty)$.  On the other hand, for $0<r<b(t)^{-1}e^{-\phi(t)+N}$, we have
\begin{align*}
\frac{\partial \overline{v}}{\partial t}\,&\geq\,\frac{-\frac{1}{3}r(x)\phi^{\frac{-2}{3}}\phi'+\varepsilon(t)\phi'(t)}{b(t) r(x)+\varepsilon(t)}\\
&\geq-\frac{\phi'}{3\phi}+\frac{\phi'}{1+bre^{\phi-N}}\geq -\frac{\phi'}{3\phi}+\frac{\phi'}{2}\\
&=(\frac{1}{2}-\frac{1}{3\phi})\phi',
\end{align*}
while for $b^{-1}e^{-\phi(t)+N} < r\leq \delta_2$,
\begin{align*}
\frac{\partial \overline{v}}{\partial t}&\geq \frac{-\frac{1}{3}r\phi^{\frac{-2}{3}}\phi'}{b r+\varepsilon(t)}+\,e^{\phi(t)-N}r(x)\phi'\\
&\geq \frac{-\phi'}{3\phi}+\,e^{\phi(t)-N}r(x)\phi'\,\geq\,\frac{-\phi'}{3\phi}+\,\phi^{-\frac{1}{3}}\phi'.
\end{align*}
Since $\phi(t)\to +\infty$, by $(\ref{inequn_fastgrowthupperbound})$, we can choose $T_4>T_3$ sufficiently large so that for $(x,t)\in U_{\delta_2}\times[T_4,+\infty)$, we have that
\begin{align*}
\overline{v}_t\geq e^{-2\overline{v}}(\Delta_g\overline{v}-K_h)-1.
\end{align*}
Moreover, we require $T_4$ be large so that $v$ is sufficiently close to $v_{LN}=u_{LN}-u_N$ on $\Sigma_{\delta_2}=\{x\in U_{\delta_2}\big|r(x)=\delta_2\}$ for $t\geq T_4$, and hence, $v$ in uniformly bounded in $\Sigma_{\delta_2}\times[T_4,+\infty)$. On the other hand, we have that
\begin{align*}
-\log(b(t) \delta_2+\varepsilon(t))+\,e^{\phi(t)-N}\delta_2\geq-\log(2\phi^{\frac{1}{3}}(t)\delta_2)+e^{\phi(t)-N}\delta_2>0
\end{align*}
for $t\geq T_4$ for $T_4>T_3$ sufficiently large. Therefore, as in the proof of Theorem \ref{thm_asymptbdlowspeed1}, we choose $A>1$ and $p>1$ large so that $\overline{v}\geq v$ on $\Sigma_{\delta_2}\times[T_4,\infty) \bigcup U_{\delta_2}\times\{T_4\}$. Then, we apply Lemma \ref{lem_comparisonCDP1} on $U_{\delta_2}\times [T_4,\infty)$ to have that
\begin{align*}
\overline{v}\geq v,
\end{align*}
on $U_{\delta_2}\times [T_4,\infty)$. Since $\overline{v}= v$ on $\partial M\times [T_4,\infty)$, we have
\begin{align*}
k_{e^{2u}g}=k_{e^{2v}h}=e^{-\phi(t)+N}[\frac{\partial v}{\partial n_h}+k_h]&\geq e^{-\phi(t)+N}[\frac{\partial \overline{v}}{\partial n_h}+k_h]\\
&=e^{-\phi(t)+N}[-\frac{\partial \overline{v}}{\partial r}\big|_{r=0}+k_h]=b(t)-1+[k_h-Ap\delta^{-p-1}]e^{-\phi(t)+N}\\
&=\phi^{\frac{1}{3}}-1+[k_h-Ap\delta^{-p-1}]e^{-\phi(t)+N},
\end{align*}
on $\partial M\times[T_4,\infty)$, where $n_h$ is the outer normal vector fields on $\partial M$ on $(\overline{M},h)$.

\end{proof}

\section{Convergence of the solution to the boundary value problem $(\ref{equn_RF3-1})-(\ref{equn_ibp3-1})$}\label{section_convergenceibpgc}

\begin{thm}\label{thm_convergCO1}
Let $(\overline{M},g)$ be a compact surface with its interior $M$ and boundary $\partial M$. Assume that $u_0\in C^{2+\alpha}(\overline{M})$ and $\psi\in C^{1+\alpha,\frac{1}{2}+\frac{\alpha}{2}}(\partial M\times[0,T])$ for all $T>0$, and also the compatibility condition holds on $\partial M\times\{0\}$:
\begin{align*}
\frac{\partial}{\partial n_g}u_0+k_{g}=\psi(\cdot,0) e^{u_0}.
\end{align*}
Suppose there exist two groups of Cauchy-Dirichlet data $(u_{01},\phi_{01})$ and $(u_{02},\phi_{02})$ satisfying the condition of Theorem \ref{thm_CDPconvergnl} such that the corresponding solutions $u_1$ and $u_2$ to the Cauchy-Dirichlet problem $(\ref{equn_RF3-2})-(\ref{equn_ibp3-2})$ satisfy that
\begin{align*}
&u_{10}\leq u_0\leq u_{20}\,\,\,\,\text{on}\,\,\overline{M},\\
&k_{e^{2u_1}g}\leq \psi \leq k_{e^{2u_2}g}\,\,\,\text{on}\,\,\partial M\times[0,\infty).
\end{align*}
Then there exists a unique solution $u\in C^{2+\alpha,1+\frac{\alpha}{2}}(\overline{M}\times[0,T])$ to the initial-boundary value problem $(\ref{equn_RF3-1})-(\ref{equn_ibp3-1})$ for all $T>0$, and moreover, $u$ converges locally uniformly in $C^2$ to the Loewner-Nirenberg solution $u_{LN}$ on $M$.
\end{thm}
\begin{proof}
The problem $(\ref{equn_RF3-1})-(\ref{equn_ibp3-1})$ is a quasilinear parabolic equation with Robin boundary condition, and at $t=0$, and the equation is strictly parabolic. With the compatibility condition in the theorem, by the Inverse Function Theorem and standard methods from the parabolic equation \cite{LSU}\cite{Lieberman}, there exists $\epsilon>0$ such that there exists a unique solution $u\in C^{2+\alpha, 1+\frac{\alpha}{2}}(\overline{M}\times[0,\epsilon])$ to $(\ref{equn_RF3-1})-(\ref{equn_ibp3-1})$. Let $T_0>0$ be the largest exists time of $u$. Then by Theorem \ref{thm_comparison2}, we have
\begin{align}\label{inequn_ODcontrolledbyCD}
u_1\leq u\leq u_2,
\end{align}
for $(x,t)\in \overline{M}\times [0,T]$ for any $T<T_0$ and hence, by standard regularity theory of second order parabolic equations, we have that there exists $C=C(T_0)>0$ independent of $T$ such that
\begin{align*}
\|u\|_{C^{2+\alpha,1+\frac{\alpha}{2}}(\overline{M}\times [0,T])}\leq C.
\end{align*}
Therefore, $u$ exists for all $t>0$, and $u\in C^{2+\alpha, 1+\frac{\alpha}{2}}(\overline{M}\times[0,T])$ satisfies $(\ref{inequn_ODcontrolledbyCD})$ for all $T>0$. By Theorem \ref{thm_CDPconvergnl}, $u_1$ and $u_2$ converge to $u_{LN}$ locally uniformly in $C^2$ sense in $M$ and hence, $u$ converges to $u_{LN}$ locally uniformly in $M$. Then by standard interior estimates of second order parabolic equations, $u$ converges to $u_{LN}$ locally uniformly in $C^2$ sense in $M$. This completes the proof of the theorem.
\end{proof}
Now we have the following long time existence and convergence result.
\begin{thm}\label{thm_convergCO2}
Let $(\overline{M},g)$ be a compact surface with its interior $M$ and boundary $\partial M$. Assume that $u_{01}\in C^{2+\alpha}(\overline{M})$ and $\phi_{01}\in C^{2+\alpha,1+\frac{\alpha}{2}}(\partial M\times[0,T])$ for all $T>0$ satisfy the condition in Theorem \ref{thm_CDPconvergnl}. Moreover, we assume that $u_{01}<u_{LN}$ on $M$.  Let $u_1$ be the solution to the Cauchy-Dirichlet problem $(\ref{equn_RF3-2})-(\ref{equn_ibp3-2})$ with the initial and boundary data $u_{01}$ and $\phi_{01}$. We assume that $u_0\in C^{2+\alpha}(\overline{M})$ and $\psi\in C^{1+\alpha,\frac{1}{2}+\frac{\alpha}{2}}(\partial M\times[0,T])$ for all $T>0$, and also the compatibility condition holds on $\partial M\times\{0\}$:
\begin{align*}
\frac{\partial}{\partial n_g}u_0+k_{g}=\psi(\cdot,0) e^{u_0}.
\end{align*}
Suppose $u_0$ and $\psi$ satisfy the following:
\begin{align*}
&u_{10}\leq u_0< u_{LN}\,\,\,\,\text{on}\,\,\overline{M},\\
&k_{e^{2u_1}g}\leq \psi \leq y(t)^{\frac{1}{3}}-2,\,\,\,\text{on}\,\,\partial M\times[0,\infty).
\end{align*}
where $y(t)\in C^3([0,\infty))$ is a positive function satisfying $y'\geq 3y+1$ for $t\in[0,\infty)$.  Then there exists a unique solution $u\in C^{2+\alpha,1+\frac{\alpha}{2}}(\overline{M}\times[0,T])$ to the initial-boundary value problem $(\ref{equn_RF3-1})-(\ref{equn_ibp3-1})$ for all $T>0$, and moreover, $u$ converges locally uniformly in $C^2$ to the Loewner-Nirenberg solution $u_{LN}$ on $M$.
\end{thm}
\begin{Remark}
If in addition, $\phi_{01}$ satisfies the slow growth condition in Theorem \ref{thm_asymptbdlowspeed1} or Theorem \ref{thm_lowspeedCDdataperturb} for $t$ large, then $k_{e^{2u_1}g}$ is uniformly bounded for all $t>0$ and $k_{e^{2u_1}g}\to1$ uniformly on $\partial M$ as $t\to\infty$.
\end{Remark}
\begin{proof}
By Theorem \ref{thm_CDPconvergnl}, there exists a unique solution $u_1$ to $(\ref{equn_RF3-2})-(\ref{equn_ibp3-2})$ with the initial and boundary data $u_{01}$ and $\phi_{01}$ for all $t>0$ and $u_1 \to u_{LN}$ locally uniformly in $C^2$ sense in $M$. For a given positive function $y(t)$, we can always find a pair of Cauchy-Dirichlet data $(u_{02},\phi_{02})$ satisfying the condition in Theorem \ref{thm_CDPconvergnl}, and we also require that there exists $T_1>0$ such that for $t\geq T_1$, we have that $\phi_{02}=\phi_{02}(t)= y(t)$ on $\partial M$ and hence
\begin{align*}
\phi_{02}'(t)\geq 3\phi_{02}(t)+1.
\end{align*}
Then by Theorem \ref{thm_CDPconvergnl}, there exists a unique solution $u_2$ to $(\ref{equn_RF3-2})-(\ref{equn_ibp3-2})$ with the initial and boundary data $u_{02}$ and $\phi_{02}$ for all $t>0$ and $u_2 \to u_{LN}$ locally uniformly in $C^2$ sense in $M$. Moreover, since $u_0$ is continuous on $\overline{M}$ and $u_0<u_{LN}$ on $M$, by the proof of Theorem \ref{thm_CDPconvergnl}, Lemma \ref{lem_lowerboundbdary} and Theorem \ref{thm_Flow1convergence}, there exists $T_2>T_1$ such that $u_2(x,t)\geq u_0(x)$ for $x\in \overline{M}$ when $t\geq T_2$. 
 On the other hand, by Theorem \ref{thm_asymptbdfastspeed2}, there exists $T_4>T_2$ such that
 \begin{align*}
k_{e^{2u_2}g}=e^{-u_2}(\frac{\partial u_2}{\partial n_g}+k_g)\geq \phi_{02}^{\frac{1}{3}}-2,
\end{align*}
on $\partial M\times[T_4,\infty)$, where $n_g$ is the outer normal vector fields on $\partial M$ on $(\overline{M},g)$. Now let $v(x,t)=u_{2}(x,t+T_4)$ for $(x,t)\in \overline{M}\times[0,\infty)$. Then we have
\begin{align*}
u_0(x)<v(x,0),
\end{align*}
for $x\in \overline{M}$, and
\begin{align*}
\psi(x,t)\leq y(t)^{\frac{1}{3}}-2\leq (\phi_{02}(t+T_4))^{\frac{1}{3}}-2\leq k_{e^{v(x,t)}g},
\end{align*}
for $(x,t)\in \partial M\times[0,+\infty)$. Therefore, by Theorem \ref{thm_convergCO1} and Theorem \ref{thm_comparison2}, there exists a unique solution $u(x,t)$ to the initial-boundary value problem $(\ref{equn_RF3-1})-(\ref{equn_ibp3-1})$ for all $t\geq 0$ and
\begin{align*}
u_1(x,t)\leq u(x,t)\leq v(x,t)
\end{align*}
for $(x,t)\in \overline{M}\times[0,+\infty)$. Then as discussed in the proof of Theorem \ref{thm_convergCO1}, we have that $u(x,t)$ converges locally uniformly to $u_{LN}(x)$ in $C^2$ sense in $M$. This completes the proof of Theorem \ref{thm_convergCO2}.

\end{proof}
\begin{Remark}
The upper bound condition of the initial data in Theorem \ref{thm_convergCO2}
\begin{align*}
u_0<u_{LN}
\end{align*}
is to guarantee the existence of a solution $u_{02}$ to the Cauchy-Dirichlet problem $(\ref{equn_RF3-2})-(\ref{equn_ibp3-2})$ as an upper bound of $u$, when applying the comparison theorem. Notice that for a solution $u_2$ to the Cauchy-Dirichlet problem $(\ref{equn_RF3-2})-(\ref{equn_ibp3-2})$ which converges to $u_{LN}$ locally, $u_{2}$ will be sufficiently close to $u_{LN}$ on any compact subset of $M$ for $t$ large.
\end{Remark}

Now we are ready to prove Theorem \ref{thm_convergCO3}, which drops the upper bound condition on $u_0$ in Theorem \ref{thm_convergCO2}.

\begin{proof}[Proof of Theorem \ref{thm_convergCO3}]
The only difference of the proof here from that of Theorem \ref{thm_convergCO2} is the upper bound estimates of $u$. Let $u_{02}(x)\in C^{4,\alpha}(\overline{M})$ be a solution to the equation $(\ref{equn_LNFn2})$. We choose $\phi_{02}\in C^{4+\alpha,2+\frac{\alpha}{2}}(\partial M,[0,T'])$ for all $T'>0$ so that $\phi_{02}$ and $u_{02}$ satisfy the $C^{4+\alpha,2+\frac{\alpha}{2}}$ compatibility condition $(\ref{equn_RF3-2-1cpr1})$ and  $\frac{\partial \phi_{02}}{\partial t}(x,0)=\frac{\partial^2\phi_{02}}{\partial t^2}(x,0)=0$ for $x\in \partial M$. Moreover, we choose $\phi_{02}$ so that $\frac{\partial\phi_{02}}{\partial t}\geq0$ on $\partial M\times [0,\infty)$ and $\phi_{02}=\phi_{02}(t)=y(t+t_1)$ for $t\geq T_1$ for some $t_1\geq0$ and $T_1>0$ large. Indeed, we can first define $\phi_{02}$ for $t\geq T_1$ and for $t\geq0$ small, and then extend it. 
  Then by Theorem \ref{thm_CDPconvergnl}, there exists a unique solution $u_2$ to $(\ref{equn_RF3-2})-(\ref{equn_ibp3-2})$ with the initial and boundary data $u_{02}$ and $\phi_{02}$ for all $t>0$ and $u_2 \to u_{LN}$ locally uniformly in $C^4$ sense in $M$. Moreover, by the proof of Theorem \ref{thm_convertibp1-a}, $\frac{\partial u_{2}}{\partial t}\geq0$ on $\overline{M}\times[0,+\infty)$ and hence, $u_2(x,t)\geq u_{02}(x)$ for $(x,t)\in \overline{M}\times[0,\infty)$. The same as in the proof of Theorem \ref{thm_convergCO2}, we let $v(x,t)=u_2(x,t+T_4)$ on $\overline{M}\times[0,\infty)$ for $T_4>T_1$ sufficiently large, such that  by Theorem \ref{thm_asymptbdfastspeed2},
\begin{align*}
k_{e^{2v(x,t)}g}=e^{-v(x,t)}(\frac{\partial v}{\partial n_g}+k_g)\geq (\phi_{02}(x,t+T_4))^{\frac{1}{3}}-2,
\end{align*}
for $(x,t)\in \partial M\times[0,\infty)$, where $n_g$ is the outer normal vector fields on $\partial M$ on $(\overline{M},g)$. Let $D=\max\{\sup_{M}(u_0-u_{02}),0\}+1$. Recall that $\phi_{02}=\phi_{02}(t)=y(t+t_1)$ with some $t_1\geq 0$ for $t\geq T_4$, and hence for any $t_3>t_2>T_4$, we have
\begin{align*}
\phi_{02}(t_3)\geq e^{3(t_3-t_2)}\phi_{02}(t_2).
\end{align*}
Let $T_5\geq D$ and $\varphi(x,t)=v(x,t+T_5)+D$. Then since
\begin{align*}
\Delta_g v-K_g=e^{2v}(v_t+1)\geq e^{2v}\geq0,
\end{align*}
on $\overline{M}\times[0,\infty)$, we have
\begin{align*}
\varphi_t(x,t)=e^{-2v}(\Delta_gv-K_g)-1=e^{2D}e^{-2\varphi(x,t)}(\Delta_g\varphi(x,t)-K_g)-1\geq e^{-2\varphi(x,t)}(\Delta_g\varphi(x,t)-K_g)-1,
\end{align*}
on $\overline{M}\times[0,\infty)$. Moreover, we have
\begin{align*}
\varphi(x,0)\geq u_0(x),
\end{align*}
for $x\in \overline{M}$, and
\begin{align*}
k_{e^{2\varphi}g}(x,t)&=e^{-\varphi(x,t)}(\frac{\partial \varphi}{\partial n_g}(x,t)+k_g)\\
&=e^{-D}e^{-v(x,t+T_5)}(\frac{\partial v}{\partial n_g}(x,t+T_5)+k_g)\\
&\geq e^{-D}[(\phi_{02}(x,t+T_5+T_4))^{\frac{1}{3}}-2]\\
&\geq (\phi_{02}(x,t+T_4))^{\frac{1}{3}}-2\\
&\geq \psi(x,t)
\end{align*}
for $(x,t)\in \partial M\times[0,+\infty)$. Therefore, by Theorem \ref{thm_convergCO1}, Theorem \ref{thm_comparison2} and Remark \ref{remark_comparison},  we have that there exists a unique solution $u(x,t)$ to the initial-boundary value problem $(\ref{equn_RF3-1})-(\ref{equn_ibp3-1})$ for all $t\geq 0$, and
\begin{align*}
u_1(x,t)\leq u(x,t)\leq \varphi(x,t)
\end{align*}
for $(x,t)\in \overline{M}\times[0,+\infty)$. Indeed, based on the upper and lower bounds control on $u$, by the standard regularity theory of the parabolic equations, the solution $u\in C^{2+\alpha,1+\frac{\alpha}{2}}(\overline{M}\times[0,T'])$ exists for all time $T'>0$. Moreover, on each compact subset $U\subseteq M$, since $u_1\to u_{LN}$ uniformly, we have that $\displaystyle\liminf_{t\to+\infty}(\inf_{x\in U}(u(x,t)-u_{LN}(x)))\geq0$. On the other hand, by the same discussion as in the proof of Theorem \ref{thm_CDPconvergnl} and Theorem \ref{thm_Flow1convergence}, we have
\begin{align*}
\displaystyle\limsup_{t\to+\infty}(\sup_{x\in M}(u(x,t)-u_{LN}(x)))\leq0.
\end{align*}
 Therefore, $u(x,t)\to u_{LN}(x)$ locally uniformly in $M$. Thus, by the standard interior estimates of the parabolic equations, we have
$u(x,t)\to u_{LN}(x)$ locally uniformly in $C^2$ sense in $M$. This completes the proof of Theorem \ref{thm_convergCO3}.

\end{proof}

\section{Discussions on conformal flows with prescribed mean curvature on the boundary in high dimensions and some open problems}\label{section_furtherdiscussion}

Let $(\overline{M},g)$ be a compact Rimennian manifolds of dimension $n$ with boundary $\partial M$ and the interior $M$. The Loewner-Nirenberg problem is to find a complete conformal metric $h$ of $g$ in $M$ such that the scalar curvature $R_h=-n(n-1)$. Identically, one wants to solve the boundary value problem
\begin{align}
&-\frac{4(n-1)}{n-2}\Delta_g u +R_gu =-n(n-1)u^{\frac{n+2}{n-2}},\,\,\text{in}\,\,M,\\
&u(p)\to+\infty,\,\,\,\text{as}\,\,p\to\partial M.
\end{align}
We introduced the Cauchy-Dirichlet problem of two types of curvature flows $g_t=u(\cdot,t)^{\frac{4}{n-2}}\,g$ for $t\geq0$ in \cite{GLi} to the Loewner-Nirenberg problem, paralleling to the approach of Loewner and Nirenberg \cite{LN} and Aviles and McOwen \cite{AM}. One is the direct flow
\begin{align}
&\label{equn_scalar1}u_t=\frac{4(n-1)}{n-2}\Delta u-R_gu-n(n-1)u^{\frac{n+2}{n-2}},\,\,\,\text{in}\,\,M\times [0,+\infty),\\
&u(p,0)=u_0(p),\,\,p\in M,\\
&\label{equn_bdv1}u(q,t)=\phi(q,t),\,\,q\in\,\partial M,
\end{align}
with $\phi(\cdot,t)\to+\infty$ uniformly as $t\to+\infty$;  the other is the Yamabe flow which is conformally invariant
\begin{align}\label{equn_Yamabeflow1}
&u_t=(n-1)u^{-\frac{4}{n-2}}(\Delta_gu-\frac{n-2}{4(n-1)}(R_gu+n(n-1)u^{\frac{n+2}{n-2}}))\\
&u(p,0)=u_0(p),\,\,p\in M,\\
&\label{equn_bdv1}u(q,t)=\phi(q,t),\,\,q\in\,\partial M,
\end{align}
with $\phi(\cdot,t)=+\infty$ uniformly as $t\to+\infty$.  More generally, one can consider the Cauchy-Dirichlet problem of the conformally invariant flow
\begin{align*}
&u_t=F(u^{-\frac{4}{n-2}}(\Delta_gu-\frac{n-2}{4(n-1)}(R_gu+n(n-1)u^{\frac{n+2}{n-2}}))),
\end{align*}
for some function $F:\,[0,\infty)\to[0,\infty)$ with $F(0)=0$. In \cite{GLi2}, we introduced the Cauchy-Dirichlet problem of another conformally invariant flow to the Loewner-Nirenberg problem of the $\sigma_k$-Ricci equation for $1\leq k\leq n$ -- the $\sigma_k$-Ricci flow $g_t=e^{2u(\cdot,t)}\,g$:
\begin{align}\label{equn_sigmakRicciflow}
&2ku_t=\log(\sigma_k(\bar{\nabla}^2u))-\log(\bar{\beta}_{k,n})-2ku,\,\,\text{in}\,\,\overline{M}\times[0,+\infty),\\
&u\big|_{t=0}=u_0,\\
&u=\phi,\,\,\,\text{on}\,\,\partial M\times[0,+\infty),
\end{align}
where $\bar{\beta}_{k,n}=(n-1)^k \left(\begin{matrix}n\\ k\end{matrix}\right)$, and $\sigma_k(\bar{\nabla}^2u)$ is the $k$-th elementary symmetric function of the eigenvalues of the matrix
\begin{align}
-\text{Ric}(g_t)=\bar{\nabla}^2 u =-\text{Ric}(g) +(n-2)\nabla_g^2u+\Delta_gu \,g\,+\,(n-2)\,(|du|_g^2g-du \otimes du).
\end{align}
In particular, when $k=1$, the steady-state solution of $(\ref{equn_sigmakRicciflow})$ corresponds to the conformal metric with scalar curvature $-n(n-1)$. We showed in \cite{GLi2} that when the Dirichlet boundary data $\phi(\cdot,t)\to+\infty$ not too slowly as $t\to\infty$ with suitable initial data $u_0$, the Cauchy-Dirichlet problem of the flow $(\ref{equn_sigmakRicciflow})$ converges locally uniformly to the Loewner-Nirenberg solution $u_{LN}$ in $M$. similar convergence results were obtained in \cite{GLi} for the two types of flows. It can be generalized to some other conformally invariant flows.

It is natural to consider the initial-boundary value problem of these flows with the Dirichlet boundary condition replaced by the prescribed mean curvature on the boundary $\partial M$:
\begin{align}
H_{g_t}\big|_{\partial M}=\psi(\cdot,t)
\end{align}
on $\partial M\times[0,\infty)$, which can be seen as generalization of the prescribed geodesic curvature problem of the $2$-D normalized Ricci flow. We wonder for what initial and boundary data $(u_0,\,\psi)$ does the flow converges locally uniformly to the Loewner-Nirenberg solution $u_{LN}$, as $t\to\infty$.

The asymptotic behavior of the mean curvature $H_{g_t}$ on $\partial M$ along the Cauchy-Dirichlet problem of the flows can be discussed very similarly as the $2$-D normalized Ricci flow in this paper: when the Dirichlet data $\phi$ increases slowly to infinity, the mean curvature of $\partial M$ satisfies that $H_{g_t}\to n-1$, as $t\to\infty$; while when $\phi$ increases sufficiently fast as $t\to\infty$, the mean curvature of $\partial M$ satisfies that $H_{g_t}\to \infty$. Also, a similar comparison theorem as Theorem \ref{thm_comparison2} can be obtained by a careful choice of the class of initial-boundary data $(u_0,\,\psi)$ and hence, the long time existence and the uniformly local convergence to the Loewner-Nirenberg solution of the solution to the initial-boundary value problem of the corresponding flows might be obtained in a parallel approach as the $2$-D normalized Ricci flow here.

On a compact $2$-D surface $\overline{M}$ with boundary $\partial M$ and its interior $M$, for the normalized Ricci flows $(\ref{equn_RF3-1})$ on $\overline{M}$ which converges locally uniformly to the Loewner-Nirenberg metric (the complete hyperbolic metric) in $M$, there exists a $1$-$1$ correspondence of the Dirichlet boundary data $\phi=u\big|_{\partial M}$ and the geodesic curvature $\psi$ on $\partial M$ for $t\in[0,\infty)$. So it is natural to use the auxiliary Cauchy-Dirichlet problem to describe the solution to $(\ref{equn_RF3-1})-(\ref{equn_ibp3-1})$. The question is: Is there a more intrinsic way to prescribe the boundary data $\psi$ in $(\ref{equn_RF3-1})-(\ref{equn_ibp3-1})$ in order that the solution $u$ converges to the Loewner-Nirenberg solution? Especially, what is the clearer description of the set of the boundary data $\psi$ with $\psi\to1$ as $t\to\infty$, in order that the solution $u$ converges to the Loewner-Nirenberg solution?

\begin{appendix}
\section{}\label{Appendix1}
Now we consider the boundary data $\tilde{\phi}$ chosen in the proof of Theorem \ref{thm_Flow1convergence}. Recall that $\phi(x,t)\in C^{2+\alpha,1+\frac{\alpha}{2}}(\partial M\times[0,\tau])$ for any $\tau>0$, and $\frac{\partial \phi}{\partial t}(x,t)\geq0$ for $(x,t)\in\partial M\times[T,\infty)$ for some $T>0$, and $\phi(x,t)\geq \log(\xi(t))$ for $(x,t)\in \partial M\times [T_1,\infty)$ with $\xi(t)$ low-speed increasing satisfying $(\ref{inequn_lowspeedinc})$ and some $T_1>0$.  Define the function
\begin{align*}
f(x,t)=\phi(x,t)-1+\frac{1}{100}(1-e^{-t})
 \end{align*}
 on $\partial M\times[T,\infty)$. Hence, $\frac{\partial f}{\partial t}(x,t)\geq \frac{1}{100}e^{-t}>0$ on $\partial M\times[T,\infty)$.  We now want to use the density of $C^5$ functions on $C^{2,1}(\partial M\times[T,\infty))$ to find a function $\tilde{\phi}\in C^5(\partial M\times[T,\infty))$ sufficiently close to the function $f(x,t)$ in $C^1(\partial M\times[T,\infty))$ such that
 \begin{align*}
 \phi-2<\tilde{\phi}<\phi,
 \end{align*}
 and $\tilde{\phi}_t>0$ on $\partial M\times[T,\infty)$. This is an exercise in calculus.

Notice that $\phi\to \infty$ uniformly on $\partial M$ as $t\to \infty$. Now we choose $T_0=T,\,T_1>T_0+10000,\,...,\,T_{k+1}>T_k+10000,...$ for $k\in \mathbb{N}$. By definition, $f_t(x,t)\geq \frac{e^{-T_{k+1}-1}}{100}$ for $(x,t)\in \partial M\times[T_k,T_{k+1}+1]$.

Let $k\in \mathbb{N} \bigcup \{0\}$. {\bf Claim}: On each $[T_k,T_{k+1}+1]$, we can find $\theta_k(x,t)$ such that $\theta_k$ sufficiently close to $f$ in $C^1(\partial M\times[T_k,T_{k+1}+1])$, and $\frac{\partial}{\partial t}\theta_k>0$ in $\partial M\times[T_k,T_{k+1}+1]$, and also \begin{align*}
\theta_k(x,t)>f(x,t)-\frac{e^{-T_k-1}}{200},
 \end{align*}
 for $(x,t)\in \partial M\times[T_k,T_{k}+1]$; while
 \begin{align*}
 \theta_k(x,t)<f(x,t)-\frac{e^{-T_{k+1}-1}}{200},
  \end{align*}
  for $(x,t)\in \partial M\times[T_{k+1},T_{k+1}+1]$. Indeed, in $\partial M\times[T_k,T_{k}+1]$, choose $\epsilon_k>0$ to be sufficiently small, and $\phi_k(x,t)\in C^5(\partial M\times[T_k,T_{k+1}+1])$ is in the $\epsilon_k$-neighborhood of $f$ in $C^1(\partial M\times[T_k,T_{k+1}+1])$ ( $\phi_k$ is almost $f$ in $C^1$ sense). We take $\epsilon_k<10^{-6}e^{-T_{k+1}-1}$. For $(x,t)\in \partial M\times[T_{k},T_{k+1}+1]$, let $\sigma_k(t)\in C^5([T_{k},T_{k+1}+1])$ such that $\sigma_k(t)=0$ for $t\leq T_{k+1}-9$, and also,
  \begin{align*}
  0\leq \sigma_k'(t)<\frac{e^{-T_{k+1}-1}}{200},
   \end{align*}
   for $t\in[T_{k+1}-10,T_{k+1}+1]$, and $\sigma_k(t)\in [\frac{3e^{-T_{k+1}-1}}{400},\frac{e^{-T_{k+1}-1}}{100}]$ for $t\in[T_{k+1},T_{k+1}+1]$. Let \begin{align*}
   \theta_k(x,t)=\phi_k(x,t)-\sigma_k(x,t),
    \end{align*}
    on $\partial M\times[T_k,T_{k+1}+1]$. Then $\frac{\partial}{\partial t}\theta_k>0$ on $\partial M\times[T_k,T_{k+1}+1]$. Then $\theta_k$ realizes the function in the {\bf Claim} for any $k\in \mathbb{N}\bigcup \{0\}$. Therefore, $\theta_k(x,t)\leq \theta_{k+1}(x,t)$ for $(x,t)\in \partial M\times[T_{k+1},T_{k+1}+1]$ for each $k\in \mathbb{N} \bigcup \{0\}$. Let $\eta_k(t)\in C^5([T_k,T_{k+2}])$ such that $\eta_k(t)=0$ for $t\leq T_{k+1}$ and $\eta_k(t)=1$ for $t\geq T_{k+1}+1$, and $\eta'(t)\geq 0$ for $t\in[T_k,T_{k+1}+2]$ for $k\in \mathbb{N}$.

Let $\tilde{\phi}(x,t)=\theta_0(x,t)$ for $(x,t)\in \partial M\times[T_0,T_0+1]$. Now we let $\tilde{\phi}(x,t)=\theta_k(x,t)$ for $(x,t)\in \partial M\times[T_k+1,T_{k+1}]$, and \begin{align*}
\tilde{\phi}(x,t)=\theta_k(x,t)(1-\eta_k(t))+\theta_{k+1}(x,t)\eta_k(t)
 \end{align*}
 for $(x,t)\in \partial M\times[T_{k+1},T_{k+1}+1]$, for each $k\in \mathbb{N}\bigcup\{0\}$. 

Therefore, $\tilde{\phi}$ is the required function: $\tilde{\phi}\in C^5(\partial M\times[T,\infty))$ close to $f(x,t)$ in $C^1(\partial M\times[T,\infty))$ such that $\phi-2<\tilde{\phi}<\phi$ and $\tilde{\phi}_t>0$ on $\partial M\times[T,\infty)$. The key point here is that
\begin{align*}
\tilde{\phi}_t'(x,t)=\frac{\partial}{\partial t}\theta_k(x,t)(1-\eta_k(t))+\frac{\partial}{\partial t}\theta_{k+1}(x,t)\eta_k(t)+\eta_k'(t)(\theta_{k+1}(t)-\theta_k(t))>0,
 \end{align*}
 for $(x,t)\in \partial M\times[T_{k+1},T_{k+1}+1]$ for $k\in \mathbb{N} \bigcup\{0\}$.

\end{appendix}


\begin{thebibliography}{s2}


\bibitem{ACF} L. Andersson, P. Chru$\acute{\text{s}}$ciel, H. Friedrich, {\em On the regularity of solutions to the Yamabe equation
and the existence of smooth hyperboloidal initial data for Einsteins field equations}, Comm. Math. Phys., 149(1992), 587-612.


\bibitem{AM} P. Aviles, R. C. McOwen, {\em  Complete conformal metrics with negative scalar curvature in compact Riemannian manifolds}, Duke Math. J. 56(2) (1988) 395-398.


\bibitem{AM1} P. Aviles, R. C. McOwen, {\em  Conformal deformation to constant negative scalar curvature on
noncompact Riemannian manifolds}, J. Diff. Geom., {\bf{27}} (1988), 225-239.

\bibitem{BE} C. Bandle, M. Ess$\acute{\text{e}}$n, {\em  On the solution of quasilinear elliptic problems with boundary blow-up}, Symposia Math., {\bf{35}} (1994), 93-111.


\bibitem{Bieberbach} L. Bieberbach, {\em  $\Delta u = e^u$ und die automorphen funktionen}, Math. Ann., {\bf{77}} (1916), 173-212.


\bibitem{Brendle1} S. Brendle, {\em  Curvature flows on surfaces with boundary}, Math. Ann., {\bf{324}} (2002), no. 3, 491-519.


\bibitem{Brendle2} S. Brendle, {\em  A family of curvature flows on surfaces with boundary}, Math. Z., {\bf{241}} (2002), no. 4, 829-869.

\bibitem{Chow} B. Chow, {\em  The Ricci flow on the $2$-sphere}, J. Diff. Geom., {\bf{33}} (1991), 325-334.

\bibitem{TChow} T.A. Chow, {\em  Ricci flow on manifolds with boundary with arbitrary initial metric}, J. Reine Angew. Math., {\bf{783}} (2022), 159-216.

\bibitem{Cortissoz} J.C. Cortissoz, {\em  Three-manifolds of positive curvature and convex weakly umbilic boundary}, Geom. Dedicata 138 (2009), 83-98.

\bibitem{CM} J.C. Cortissoz, A. Murcia, {\em  The Ricci flow on surfaces with boundary}, Comm. Anal. Geom. 27 (2019), no. 2, 377-420.

\bibitem{Gianniotis1} P. Gianniotis, {\em  Boundary estimates for the Ricci flow}, Calc. Var. Partial Differential Equations 55 (2016), no. 1, Art. 9, 21 pp.

\bibitem{Gianniotis2} P. Gianniotis, {\em  The Ricci flow on manifolds with boundary}, J. Differential Geom. 104 (2016), no. 2, 291-324.

\bibitem{Hamilton1} R. Hamilton, {\em  Three-manifolds with positive Ricci curvature}, J. Diff. Geom., {\bf{17}} (1982), 255-306.

\bibitem{Hamilton2} R. Hamilton, {\em  The Ricci flow on surfaces}, Mathematics and General Relativity, Contemporary
Math., Vol. 71, Amer. Math. Soc, Providence, RI, 1988, 237-262.

\bibitem{Hamilton}  R. Hamilton, {\em  Four-manifolds with positive curvature operator}, J. Diff. Geom., {\bf{24}} (1986), 153-179.

\bibitem{HJ} Q. Han, X. Jiang, {\em  Boundary regularity of minimal graphs in the hyperbolic space}, J. Reine Angew. Math. 801 (2023), 239-272.

\bibitem{HS} Q. Han, W. Shen, {\em  Boundary behaviors for Liouville's equation in planar singular domains}, J. Funct. Anal. 274 (2018), no.6, 1790-1824.

\bibitem{Jouttijarvi}  R. Jouttij$\ddot{\text{a}}$rvi, {\em  Novel Boundary Conditions for the Ricci Flow}, preprint, arXiv: 2403.09169v1.

\bibitem{Kichenassamy} S. Kichenassamy, {\em Boundary behavior in the Loewner-Nirenberg problem}, J. Funct. Anal. {\bf 222} (2005), 98-113.

\bibitem{KS} K. Kunikawa, Y. Sakurai, {\em  Hamilton type entropy formula along the Ricci flow on surfaces with boundary}, Comm. Anal. Geom. 31 (2023), no. 7, 1655-1668.

\bibitem{LSU} O. A. Ladyzenskaya, V. A. Solonnikov, and N. N. Uraltceva, {\em  Linear and Quasilinear Equations of Parabolic Type}, Transl. Math. Mono. 23, Amer. Math. Soc., Providence, R.I., (1968).

\bibitem{LM} A.C. Lazer, P.J. McKenna, {\em  On a problem of Bieberbach and Rademacher}, Nonlinear Anal. 21 (1993), 327-335.

\bibitem{LM2} A.C. Lazer, P.J. McKenna, {\em  Asymptotic behavior of boundary blow-up problems}, Differential and
Integral Equations 7 (1994), 1001-1019.

\bibitem{GLi} G. Li, {\em Two flow approaches to the Loewner-Nirenberg problem on manifolds}, J. Geom. Anal. {\bf 32} (2022), no.1, Paper No.7, 33 pp.

\bibitem{GLi2} G. Li, {\em A flow approach to the generalized Loewner-Nirenberg problem of the $\sigma_k$-Ricci equation}, Calc. Var. 61, 169 (2022). https://doi.org/10.1007/s00526-022-02283-8.

\bibitem{TLi} T. Li, {\em  The Ricci flow on surfaces with boundary}, Thesis (Ph.D.)-University of California, San Diego.

\bibitem{Lieberman} G.M. Lieberman, {\em  Second order parabolic differential eqautions}, World Scientific Publishing Co., Inc., River Edge, NJ. ISBN: 981-02-2883-X (1996).

\bibitem{LN} C. Loewner, L. Nirenberg, {\em Partial differential equations invariant under conformal or
projective transformations}, Contributions to Analysis, Academic Press, New York, 1974, pp. 245-272.

\bibitem{Mazzeo} R. Mazzeo, {\em  Regularity for the singular Yamabe problem}, Indiana Univ. Math. J., {\bf{40}} (1991), 1277-1299.

\bibitem{Pulemotov} A. Pulemotov, {\em Quasilinear parabolic equations and the Ricci flow on manifolds with boundary}, J. Reine Angew. Math. 683 (2013), 97-118.


\bibitem{Shen} Y. Shen, {\em On Ricci deformation of a Riemannian metric on manifold with boundary}, Pacific J. Math. 173 (1996), no. 1, 203-221.

\bibitem{Zhang} H. Zhang, {\em Evolution of curvatures on a surface with boundary to prescribed functions}, Manuscripta Math. 149 (2016), no. 1-2, 153-170.


\end{thebibliography}
 \end{document}